\documentclass[a4paper, 10pt]{article}
\usepackage[english]{babel}
\usepackage{amsmath, amsthm, amsfonts, amssymb, accents}
\usepackage{graphicx, dsfont, srcltx, color}

\usepackage{srcltx}

\numberwithin{equation}{section}
\theoremstyle{plain}
\newtheorem{theorem}{Theorem}[section]
\newtheorem{lemma}{Lemma}[section]
 \newtheorem{remark}{Remark}[section]

\def\Om{\Omega}
\def\om{\omega}

\def\g{\gamma}
\def\l{\lambda}
\def\p{\partial}

\def\a{\alpha}
\def\b{\beta}

\def\d{\delta}
\def\L{\Lambda}
\def\z{\zeta}

\def\vp{\varphi}
\def\vr{\varrho}
\def\vt{\vartheta}
\def\vk{\varkappa}
\def\Ho{\mathring{W}_2}
\def\Holoc{\mathring{W}_{2,loc}}

\def\iu{\mathrm{i}}
\def\di{\,d}

\def\Op{\mathcal{H}}

\def\tE{\tilde{E}}
\def\tPsi{\tilde{\Psi}}

\def\la{\langle}
\def\ra{\rangle}

\def\cR{\mathcal{R}}
\def\cP{\mathcal{P}}
\def\cC{\mathcal{C}}

\def\cL{\mathcal{L}}
\def\cT{\mathcal{T}}
\def\cS{\mathcal{S}}
\def\cQ{\mathcal{Q}}
\def\cI{\mathcal{I}}

\def\cF{\mathcal{F}}
\def\cU{\mathcal{U}}

\def\sAo{\mathring{\mathsf{A}}}

\def\cAo{\mathring{\mathrm{A}}}

\def\sC{\mathsf{C}}
\def\sB{\mathsf{B}}
\def\sA{\mathsf{A}}

\def\sE{\mathsf{E}}

\def\H{W_2}
\def\Hloc{W_{2,loc}}

\def\Hoper{\mathring{W}_{2,per}}

\def\mr{\mathfrak{r}}
\def\mW{\mathfrak{W}}
\def\mV{\mathfrak{V}}

\def\mG{\mathfrak{G}}
\DeclareMathOperator{\RE}{Re}
\DeclareMathOperator{\IM}{Im}
\DeclareMathOperator{\spec}{\sigma}

\DeclareMathOperator{\essspec}{\sigma_{e}}

\begin{document}

\title{On finitely many resonances emerging under distant perturbations in multi-dimensional cylinders}

\author
{D.I. Borisov$^1$\footnote{Corresponding author}, A.M. Golovina}

\vskip -0.5 true cm

\date{\empty}

\maketitle

\begin{center}
{\footnotesize $^1$
Institute of Mathematics, Ufa Federal Research Center, Russian Academy of Sciences, Ufa, Russia,
\\
Bashkir State University, Ufa, Russia,
\\
University of Hradec Kr\'alov\'e,  Hradec Kr\'alov\'e, Czech Republic
 \\
{\tt borisovdi@yandex.ru}
\\
\footnotesize $^2$
Bauman Moscow State Technical University
\\[-.3em]
{\tt amgolovina@bmstu.ru}
}
\end{center}

\date{\empty}

\maketitle

\allowdisplaybreaks

\begin{abstract}
We consider a general elliptic operator in an infinite multi-dimensional cylinder with several distant perturbations; this operator is obtained by ``gluing''  several single perturbation operators $\Op^{(k)}$, $k=1,\ldots,n$, at large distances. The coefficients of each  operator $\Op^{(k)}$ are periodic in the outlets of the cylinder; the structure of these periodic parts at different outlets can be different. We consider a point $\l_0
\in\mathds{R}$ in the essential spectrum of the operator with several distant perturbations and assume that this point is not in the essential spectra of middle operators $\Op^{(k)}$, $k=2,\ldots,n-1$, but is an eigenvalue of at least one of $\Op^{(k)}$, $k=1,\ldots,n$. Under such assumption we show that the operator with several distant perturbations possesses finitely many resonances in the vicinity of $\l_0$. We find the leading terms in asymptotics for these resonances, which turn out to be exponentially small. We also conjecture that the made assumption selects the only case, when  the distant perturbations produce finitely many resonances in the vicinity of $\l_0$. Namely, as $\l_0$ is in the essential spectrum of at least one of operators $\Op^{(k)}$, $k=2,\ldots,n-1$, we do expect that infinitely many resonances emerge in the vicinity of $\l_0$.
\end{abstract}

\section{Introduction}

Resonances of unbounded self-adjoint operators are
intensively studied both in mathematics and physics. In mathematics, the notion of resonance is related with an analytic continuation of the resolvent of a considered operator through the essential spectrum and the resonances are defined as poles of such continuation, see, for instance, monographs \cite{Yaf}, \cite{Zwo}.

Behavior of resonances is studied in various aspects and one of interesting directions in mathematical physics is on problems with distant perturbations.  A classical example is a Schr\"odinger operator with two wells  separated by a large distance. A more general case is an elliptic differential operator in an unbounded domain, whose coefficients can be shortly described as several localized profiles connected by periodic backgrounds, see Figures~1,~2. The behavior of the resolvents and isolated eigenvalues of such operators in various particular cases were studied in a series of works, see, for instance,
\cite{AKM}, \cite{JPA04}, \cite{JMP06}, \cite{D},   \cite{GHS}, \cite{H}, \cite{HK}, \cite{BoDAN06}. General results on the resolvents and isolated eigenvalues of general operators with abstract distant perturbations were established in \cite{AHP07}, \cite{MPAG07}, \cite{GolRJMP}, \cite{GolAA}, \cite{GolPMA}.

Resonances of the operators with distant perturbations were studied much less. A classical case of the Schr\"odinger operator with two or several well separated by large distances was considered in \cite{HoMeb}, \cite{KS1}. It was found that once the Schr\"odinger operator with one of the wells has a virtual level at the bottom of its essential spectrum, the multiple well operator has infinitely many resonances near the same bottom; the leading terms in the asymptotics for these resonances were found. In \cite{Bar},  a one-dimensional  Schr\"odinger operator was considered with a periodic truncated potential, that is, a periodic potential replaced by the zero as $|x|>L$ for $L$  large enough. This is also a model with distant perturbation since such operator can be regarded as glued from two discrete Schr\"dinger operators with a potential being periodic as $\pm x>0$ and vanishing as $\mp x<0$. A similar discrete model was studied in \cite{Kl}. The main result in \cite{Bar}, \cite{Kl} stated that both discrete and continuous operators with truncated periodic potentials had a growing number of closely spaced resonances accumulated along some curve near the bottom of the essential spectrum. This result is very similar to that of \cite{HoMeb}, \cite{KS1} with the only difference that the presence of the virtual level was replaced by the truncation of the periodic potential. Very recently, a similar phenomenon was studied in \cite{JPA19}, \cite{AML20} again for a one-dimensional operator with two localized perturbations separated by a large distance. The perturbations were differential operators with compactly supported coefficients. An important feature was that these perturbations were not supposed to be symmetric in the operator sense, so, the perturbed operator could be non-self-adjoint. No presence of the virtual level at the bottom of the essential spectrum or the truncation of a periodic potential were assumed. It was shown that nevertheless, the phenomenon still held. Namely,  there emerges a growing number of resonances or eigenvalues accumulating to a fixed segment in the essential spectrum. The presence of the eigenvalues is due to non-self-adjointness of the considered operator. The properties of these emerging eigenvalues and resonances were studied in much more details than in \cite{Bar}, \cite{HoMeb}, \cite{KS1}, \cite{Kl}. In particular, these  eigenvalues and resonances were found  as sums of explicitly written convergent series and the error terms in these series were estimated   uniformly in an index counting these resonances and eigenvalues; these series also served as asymptotic ones. Also a simple effective procedure was proposed for finding these eigenvalues and resonances numerically with an arbitrary high precision.

Operators with distant perturbations not necessarily always have infinitely many resonances accumulating near some point or a segment in the essential spectrum. They can be situations, when only finitely many resonances emerge from a point in the essential spectrum. This was the case in \cite{BEG}, where a Laplacian was considered in the strip with a combination of Dirichlet and Neumann conditions imposed so that the final model could be regarded as an operator with three distant perturbations and the central perturbation had a simple isolated discrete eigenvalue $\l_0$ embedded into the essential spectrum of the left and right perturbations. It was shown that the original operator had one resonance converging to $\l_0$ and the leading terms in its asymptotics were found.

In the present work we consider a general second order scalar differential operator in a multi-dimensional cylinder with finitely many distant perturbations. This operator is obtained via ``gluing'' several single perturbation operators, each having coefficients periodic outside some compact set, see Figure~1. We study the resonances of such operator with distant perturbations emerging in the vicinity of a some point $\l_0$ in its essential spectrum. This point is supposed to be an internal in the essential spectra of the left or right single perturbation operator, not to belong to the essential spectra of middle single perturbations operators and to be an isolated eigenvalue of at least one of the single perturbation operators. Under this assumption we show that the considered operator with distant perturbations has finitely many resonances in the vicinity of $\l_0$. We find the leading terms of their asymptotics expansions; our technique also allows one to construct the next terms in these asymptotics, although this is quite a bulky and technical procedure. The situation we consider is important and deserves an independent study  since it seems to be the only case, when finitely many resonances emerge in the vicinity of the point $\l_0$. Once our assumption on $\l_0$ fails, namely, if $\l_0$ is in the essential spectrum of one of middle single perturbation operators, we strongly expect that infinitely many resonances should emerge in the vicinity of $\l_0$ similar to the model studied in \cite{JPA19}, \cite{AML20}.

\section{Problem and results}

\subsection{Problem}\label{ss:pr}

Let $x=(x_1,x')$, $x'=(x_2,\ldots,x_d)$ be Cartesian coordinates in $\mathds{R}^d$ and $\mathds{R}^{d-1}$, respectively, $d\geqslant 2$, $\om\subset\mathds{R}^{d-1}$ be a bounded domain with a boundary in the class $C^2$, $\Pi:=\mathds{R}\times\om$  be an infinite straight cylinder in  $\mathds{R}^d$. By $A_{ij}^{(k)}=A_{ij}^{(k)}(x)$,
$A_j^{(k)}=A_j^{(k)}(x)$, $A_0^{(k)}=A_0^{(k)}(x)$, $i,j=1,\ldots,d$, $k=1,\ldots,n$, $n\geqslant 2$, we denote real functions possessing the following smoothness: $A_{ij}^{(k)},\, A_j^{(k)}\in W_\infty^1(\Pi)$, $A_0^{(k)}\in L_\infty(\Pi)$. The functions $A_{ij}^{(k)}$ are assumed to be symmetric and to satisfy an ellipticity condition:
\begin{equation*}%\l%abel{2.1}
A_{ij}^{(k)}(x)=A_{ji}^{(k)}(x),\quad \sum\limits_{i,j=1}^{d} A_{ij}^{(k)}(x)\z_i\overline{\z_j}\geqslant c_0\sum\limits_{j=1}^{d} |\z_j|^2,\quad x\in\overline{\Pi},\quad \z_i\in\mathds{C},
 \end{equation*}
where $c_0>0$ is some fixed constant independent of $x$ and $\z$.

We suppose that the functions    $A_{\natural}^{(k)}$, $\natural=i,j;\, j;\, 0$,
are periodic in the outlets of the cylinder $\Pi$. Namely, there exists  a constant $a>0$ such that as $\pm x_1\geqslant a$, the functions  $A_{\natural}^{(k)}$ are periodic in $x_1$, that is,
%\begin{equation}\l%abel{2.3}
\begin{gather*}
A_{\natural}^{(k)}(x_1-p T^{(k-1)},x')=A_{\natural}^{(k)}(x),\quad \hphantom{^{+1}} x_1\leqslant -a,\qquad A_{\natural}^{(k)}(x_1+p T^{(k)},x')=A_{\natural}^{(k)}(x),\quad x_1\geqslant a,
\\
\natural=i,j;\, j;\, 0;\quad i,j=1,\ldots,d,\quad j=1,\ldots,d,\quad k=1,\ldots,n,\quad p\in\mathds{N},
\end{gather*}
%\end{equation}
where $T^{(k)}>0$, $k=0,\ldots,n$, are some periods. We also assume that
\begin{equation*}%\l%abel{2.4}
A_{\natural}^{(k)}(a^{(k)}_+ + t,x')=A_{\natural}^{(k+1)}(-a^{(k+1)}_- + t, x'), \qquad \natural=i,j;\,j;\,0,\quad x'\in\om,\quad 0\leqslant t\leqslant T^{(k)},
\end{equation*}
where $a_\pm^{(k)}\geqslant a$  are some numbers, $i,j=1,\ldots,d$, $k=1,\ldots,n-1$.
This condition means that the right periodic part of the function  $A_{\natural}^{(k)}$ coincides identically with the left periodic part of the function   $A_{\natural}^{(k+1)}$, see a schematic demonstration on Figure~1.

By  $X^{(k)}$, $k=1,\ldots,n$,  we denote a set of the numbers obeying the conditions
\begin{equation*}%\l%abel{2.5}
X^{(k+1)}_\ell-X^{(k)}_\ell=a_-^{(k+1)}+a_+^{(k)} + (\ell_-^{(k)}+\ell^{(k)}_+)T^{(k)}, \quad k=1,\ldots,n-1,
\end{equation*}
where $\ell_\pm^{(k)}$ are some arbitrary natural numbers. Such numbers always exist since it is sufficient to choose  $X^{(1)}_\ell$ and to find other  $X^{(k)}_\ell$ by the above conditions. Let  $\chi_\ell^{(1)}(x)$  be the characteristic function of the semi-infinite  interval $(-\infty,a_+^{(1)}+\ell_+^{(1)}T^{(0)}_\ell]$, and $\chi_\ell^{(k)}$, $k=2,\ldots,n-1$, be the characteristic functions of the intervals $(X^{(k-1)}-a_-^{(k)}-\ell_-^{(k-1)} T^{(k-1)}_\ell, X^{(k)}+a_+^{(k)}+\ell_+^{(k)}T^{(k)}_\ell]$, and finally, $\chi_\ell^{(n)}$ be the characteristic function of the semi-infinite interval $[X^{(n)}-a_-^{(n)}-\ell_-^{(n)}T^{(n-1)}_\ell,+\infty)$.
We denote:
\begin{equation*}
A_\natural^{(\ell)}(x):=\sum\limits_{k=1}^{n} \chi_\ell^{(k)}(x) A_\natural^{(k)}(x_1-X^{(k)}_\ell,x').
\end{equation*}
A schematic graph of these functions is provided on Figure~2.

The main object of our study is the operator in $L_2(\Pi)$ acting as
\begin{equation}\label{2.6}
\Op_\ell= -\sum\limits_{i,j=1}^{d} \frac{\p\ }{\p x_i} A_{ij}^{(\ell)}(x)\frac{\p\  }{\p x_j} + \iu \sum\limits_{j=1}^{d} \left(A_j^{(\ell)}(x)\frac{\p\ }{\p x_j} + \frac{\p\ }{\p x_j} A_j^{(\ell)}(x)\right) + A_0^{(\ell)}(x),
\end{equation}
on the domain  $\Ho^2(\Pi)$, which is a subspace of the Sobolev space  $\H^2(Q)$ consisting of the functions with  the zero trace on  $\p\Pi$, while   $\ell:=(\ell_\pm^{(k)})_{k=1,\ldots,n-1}$ and $\iu$ is the imaginary unit. The operator  $\Op_\ell$ is self-adjoint and lower-semibounded. We consider the case $\ell_\pm^{(k)}\to+\infty$, $k=1,\ldots,n-1$; in what follows, we briefly write this fact as   $\ell\to\infty$.

\begin{figure}
\begin{center}
\includegraphics[scale=0.65]{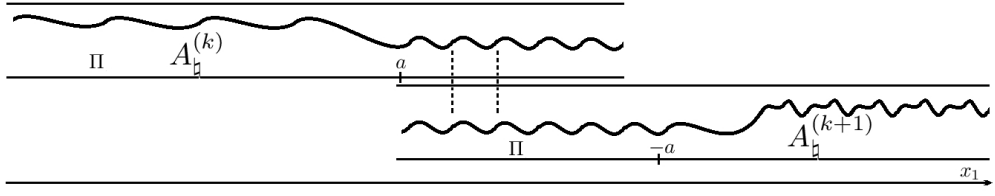}

\caption{\small Schematic graphs of the coefficients $A_\natural^{(k)}$ shown in two copies of the cylinder $\Pi$.}
\end{center}
\end{figure}

\begin{figure}
\begin{center}
\includegraphics[scale=0.65]{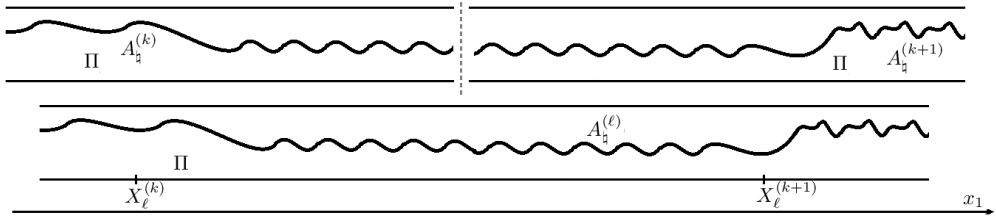}

\caption{\small Schematic graphs of the coefficients $A_\ell^{(k)}$. In the upper part of the figure two copies of the cylinder  $\Pi$ are shown and in each of them we provide the sketches of the graphs of the functions  $A_\natural^{(k)}$ and $A_\natural^{(k+1)}$. In the lower part of the cylinder $\Pi$ we show the graph of the function $A_\natural^{(\ell)}$ glued from the functions  $A_\natural^{(k)}$ and $A_\natural^{(k+1)}$.}
\end{center}
\end{figure}

Our main aim is to study the resonances of the operator  $\Op_\ell$ emerging from certain internal points in the essential spectrum as   $\ell \to \infty$, $k=1,\ldots,n-1$. The resonances are introduced as poles of an appropriate local analytic continuation of the resolvent of the operator $\Op_\ell$ in the vicinity of a point $\l_0\in\mathds{R}$.

Our main assumption on the point $\l_0$ is formulated in terms of a family of self-adjoint lower semi-bounded operators
in the space  $L_2(\Pi)$ acting as
\begin{equation*}%\l%abel{2.2}
\Op^{(k)}:=-\sum\limits_{i,j=1}^{d} \frac{\p\ }{\p x_i} A_{ij}^{(k)}(x)\frac{\p\  }{\p x_j} + \iu \sum\limits_{j=1}^{d} \left(A_j^{(k)}(x)\frac{\p\ }{\p x_j} + \frac{\p\ }{\p x_j} A_j^{(k)}(x)\right) + A_0^{(k)}(x)
\end{equation*}
on the domain  $\Ho^2(\Pi)$. This assumption says that $\l_0$ is an internal point of the essential spectrum of $\Op_\ell$, is not in the essential spectra of the operators $\Op^{(k)}$, $k=2,\ldots,n-1$, but is an eigenvalue of at least one  of the operators  $\Op^{(k)}$, $k=1,\ldots,n$. In the next subsection we describe the aforementioned local analytic continuation  of the resolvent of the operator $\Op_\ell$ in the vicinity of a point $\l_0\in\mathds{R}$.

\subsection{Analytic continuation}

We first introduce some auxiliary notations and mention known facts. We denote
\begin{align*}
&A_{\natural,per}^{(0)}(x):=A_\natural^{(1)}(x),\qquad x\in[-a_-^{(1)}-T^{(0)},-a_-^{(1)}]\times\overline{\om},
\\
&A_{\natural,per}^{(k)}(x):=A_\natural^{(k)}(x), \qquad x\in[a_+^{(k)},a_+^{(k)}+T^{(k)}]\times\overline{\om},\quad k=1,\ldots,n.
\end{align*}
We continue the functions  $A_{\natural,per}^{(k)}(x)$ $T^{(k)}$-periodically in $x_1$ on entire cylinder  $\Pi$ and consider the family of self-adjoint lower-semibounded periodic operators
\begin{equation*}%\l%abel{2.7}
\Op^{(k)}_{per}:=-\sum\limits_{i,j=1}^{d} \frac{\p\ }{\p x_i} A_{ij,per}^{(k)}(x)\frac{\p\  }{\p x_j} + \iu \sum\limits_{j=1}^{d} \left(A_{j,per}^{(k)}(x)\frac{\p\ }{\p x_j} + \frac{\p\ }{\p x_j} A_{j,per}^{(k)}(x)\right) + A_{0,per}^{(k)}(x)
\end{equation*}
in $L_2(\Pi)$ on the domain $\Ho^2(\Pi)$. The spectra of these operators have a band structure:
\begin{equation*}%\l%abel{2.8}
\spec(\Op_{per}^{(k)})=\bigcup\limits_{p=1}^{\infty} \Big\{E_p^{(k)}(\tau):\, \tau\in\left(-\tfrac{\pi}{T^{(k)}},\tfrac{\pi}{T^{(k)}}\right]\Big\},
\end{equation*}
where $\spec(\cdot)$ is the spectrum of an operator, $E_p^{(k)}(\tau)$ are the eigenvalues  of the corresponding operator  on the periodicity cell $\square^{(k)}:=(0,T^{(k)})\times\om$ taken in the ascending order counting multiplicities. The corresponding operators on the periodicity cells are self-adjoint lower-semibounded  operators in $L_2(\square^{(k)})$
 acting as
\begin{align}\label{2.9}
&\hat{\Op}^{(k)}_{per}(\tau):=\sum\limits_{i,j=1}^{d}D_i A_{ij,per}^{(k)} D_j + \sum\limits_{j=1}^{d} \big(A_{j,per}^{(k)}D_j + D_j A_{j,per}^{(k)}\big) + A_{0,per}^{(k)},\quad \tau\in\big(-\tfrac{\pi}{T^{(k)}},\tfrac{\pi}{T^{(k)}}
\big],
\\
&D_1:=\iu\frac{\p\ }{\p x_1} - \tau,\quad D_j:=\iu\frac{\p\ }{\p x_j},\quad j=2,\ldots,d, \nonumber
\end{align}
on the domains $\Hoper^2(\square^{(k)})$, which is a subspace in the Sobolev space
$\Ho^2(\square^{(k)})$ consisting of the functions satisfying the periodic boundary conditions on the lateral sides  $\p\square^{(k)}\setminus\p\Pi$ of the cell $\square^{(k)}$ and having zero trace on $\p\Pi\cap\p\square^{(k)}$.

The following lemma, implied by  \cite[Lm. 2.1]{PMA2020}, describes the essential spectrum of the operators  $\Op^{(k)}$ and $\Op_\ell$; such spectrum is denoted by the symbol $\essspec(\cdot)$.

\begin{lemma}\label{lm2.1}
The essential spectra of the operators $\Op^{(k)}$ and $\Op_\ell$ read as
\begin{equation*}
\essspec(\Op^{(k)})=\spec(\Op^{(k-1)}_{per})\cup\spec(\Op^{(k)}_{per}),\quad k=1,\ldots,n,
\qquad \essspec(\Op_\ell)=\spec(\Op^{(0)}_{per})\cup\spec(\Op^{(n)}_{per}).
\end{equation*}
\end{lemma}

We note that the band functions $E_p^{(k)}(\tau)$  and the corresponding eigenfunctions $\Psi_p^{(k)}(x,\tau)$ can be $\frac{2\pi}{T^{(k)}}$-periodically continued in  $\tau$ on the entire real line preserving their smoothness. In what follows, for the sake of convenience, we assume that this continuation is made on the segment  $\big[-\tfrac{2\pi}{T^{(k)}},\tfrac{2\pi}{T^{(k)}}\big]$.

Since $\l_0\in\mathds{R}$ is an internal point of the essential spectra of $\Op_\ell$, by Lemma~\ref{lm2.1}, this means that  $\l_0$ is an internal point in the essential spectrum of at least one of the operators  $\Op^{(k)}_{per}$, $k\in\{0,n\}$. It follows from Theorem~3.9  in \cite[Ch. V\!I\!I, Sect. 5]{Kato} that there exist points   $\tau_p^{(k)}\in\big(-\tfrac{\pi}{T^{(k)}},\tfrac{\pi}{T^{(k)}}\big]$ and numbers $\d_0$, $\d_p^{(k)}>0$ such that the eigenvalues of the operator $\Op_{per}^{(k)}(\tau)$ located in the interval  $[\l_0-\d_0,\l_0+\d_0]$ form holomorphic in  $\tau\in[\tau_p^{(k)}-\d_p^{(k)},\tau_p^{(k)}+\d_p^{(k)}]$  branches $\tE_p^{(k)}(\tau)$, $p=1,\ldots,M^{(k)}$, where $M^{(k)}$ are some numbers. The image of each branch $\tE_p^{(k)}$ covers the segment $[\l_0-\d_0,\l_0+\d_0]$. The corresponding eigenfunctions  $\tPsi_p^{(k)}(x,\tau)$ are orthonormalized in  $L_2(\square^{(k)})$  and are holomorphic in $\tau\in[\tau_p^{(k)}-\d_p^{(k)},\tau_p^{(k)}+\d_p^{(k)}]$. The identities hold:
\begin{equation}\label{2.12}
\tE_p^{(k)}(\tau_p^{(k)})=\l_0.
\end{equation}
Our main assumption is the following identities:
\begin{equation}\label{2.13}
\frac{d \tE_p^{(k)}}{d\tau}(\tau_p^{(k)})\ne0, \qquad p=1,\ldots,M^{(k)}.
\end{equation}
We stress that we do not suppose the ascending ordering for the eigenvalues  $\tE_p^{(k)}(\tau)$, $\tau\in[\tau_p^{(k)}-\d_p^{(k)},\tau_p^{(k)}+\d_p^{(k)}]$. The derivative in the left hand side in inequality  (\ref{2.13}) is obviously real.

It is clear that the eigenvalues $\tE_p^{(k)}(\tau)$ and the corresponding eigenfunctions $\tPsi_p^{(k)}(x,\tau)$ can be holomorphically continued in $\tau$ into some neighbourhood of the segment $[\tau_p^{(k)}-\d_p^{(k)},\tau_p^{(k)}+\d_p^{(k)}]$ in the complex plane; in this sense we regard them as complex functions depending on a complex parameter $\tau$. We fix $k$, $p$ and in a small neighbourhood of the point  $\l_0$ in the complex plane we consider the equation
\begin{equation}\label{2.15}
\tE_p^{(k)}(\tau)=\l,
\end{equation}
where the parameter $\tau$ ranges in the aforementioned neighbourhood of the segment $[\tau_p^{(k)}-\d_p^{(k)},\tau_p^{(k)}+\d_p^{(k)}]$ in the complex plane. Conditions  (\ref{2.12}), (\ref{2.13}) allow us to apply the implicit function theorem to equation  (\ref{2.15}) and to conclude on the unique solvability of this equation in a small neighbourhood of the point  $\tau_p^{(k)}$. Without loss of generality we can assume that these neighbourhoods coincide with the aforementioned neighbourhood of the segments $[\tau_p^{(k)}-\d_p^{(k)},\tau_p^{(k)}+\d_p^{(k)}]$. We denote the corresponding solutions by  $\mathfrak{t}_p^{(k)}(\l)$. These functions are holomorphic in  $\l$ in a small neighbourhood of the point $\l_0$ and  satisfy the identities:
\begin{equation*}%\l%abel{2.21}
\mathfrak{t}_p^{(k)}(\l_0)=\tau_p^{(k)}.
\end{equation*}
We denote $\psi_p^{(k)}(x,\l):=e^{\iu \mathfrak{t}_p^{(k)}(\l) x_1} \tilde{\Psi}_p^{(x)}(x,\mathfrak{t}_p^{(k)}(\l))$. The symbol  $\Holoc^2(\Pi)$ stands for the space of the functions in  $\Hloc^2(\Pi)$ with a zero trace on $\p\Pi$.

The following lemma was proved in \cite{PMA2020}.

\begin{lemma}\label{lm2.2}
The functions $\psi_p^{(k)}(x,\l)$ belong to  $\Holoc^2(\Pi)$  and solve the equations
%\begin{equation*}%\l%abel{2.17}
$(\hat{\Op}_{per}^{(k)}-\l)\psi_p^{(k)}=0$ in $\Pi$. %$\quad \text{in}\quad \Pi.
%\end{equation*}
Let
\begin{equation}\label{2.19}
\frac{d \tE_p^{(k)}}{d\tau}(\tau_p^{(k)})>0.
\end{equation}
Then the functions  $\psi_p^{(k)}(x,\l)$ decay exponentially as  $x_1\to\pm\infty$ and grow exponentially as $x_1\to\mp\infty$ if $\pm\IM \l>0$. 
Let
\begin{equation}\label{2.18}
\frac{d \tE_p^{(k)}}{d\tau}(\tau_p^{(k)})<0.
\end{equation}
Then the functions  $\psi_p^{(k)}(x,\l)$  grow exponentially as $x_1\to\pm\infty$ and decay exponentially as $x_1\to\mp\infty$ if $\pm\IM \l>0$.
\end{lemma}

By $P_+$ we denote the subset in the set $\{1,\ldots,M^{(n)}\}$ of the elements, for which condition (\ref{2.19}) holds. Similarly,  $P_-$ is a subset in the set $\{1,\ldots,M^{(0)}\}$ of the elements, for which condition (\ref{2.18}) holds. If $\l_0\not\in\essspec(\Op^{(n)}_{per})$, to simplify the notations and calculations, we let $P_+:=\emptyset$, $\psi_p^{(n)}:=0$. If $\l_0\not\in\essspec(\Op^{(0)}_{per})$, we also let $P_-:=\emptyset$, $\psi_p^{(0)}:=0$.

The operators $\Op^{(k)}_{per}(\tau)$, $k\in\{0,\ldots,n\}$, introduced in (\ref{2.9}) are well-defined also for complex-valued $\tau$. At that, they are no longer self-adjoint for non-real $\tau$ but it is still closed on the same domain. We consider the operator  $\Op^{(k)}_{per}(\tau)-\l_0$ as a quadratic operator pencil with the spectral parameter $\tau$. If  $\tau$ is an eigenvalue of the operator pencil $\Op^{(k)}_{per}(\tau)-\l_0$, then $\tau+\frac{2\pi}{T^{(k)}}$ is also an eigenvalue and the corresponding eigenfunctions are related by the multiplying by  $e^{\frac{2\pi\iu}{T^{(k)}} x_1}$. This is in what follows we consider the eigenvalues of this pencils located only in the strip  $-\frac{\pi}{T^{(k)}}<\RE \tau\leqslant \frac{\pi}{T^{(k)}}$ identifying the boundaries of this strip.

If $\l_0$ is a point in the essential spectrum of the operator $\Op^{(k)}_{per}$, $k\in\{0,n\}$, by Lemma~\ref{lm2.2}, the numbers $\tau_p^{(k)}$ exhaust all real eigenvalues of the operator pencil $\Op^{(k)}_{per}(\tau)-\l_0$; the corresponding eigenfunctions are  $\psi_p^{(k)}(x,\l_0)$. If  $\l_0$ is not in the essential spectrum of the operator $\Op^{(k)}_{per}$, the operator pencil  $\Op^{(k)}_{per}(\tau)-\l_0$ possesses no real eigenvalues.

For an arbitrary $b$ we denote:
\begin{equation*}%\l%abel{2.15a}
\Pi_b:=\Pi\cap\{x:\, |x_1|<b\},\quad \Pi_b^+:=\Pi\cap\{x:\, x_1>b\},\quad \Pi_b^-:=\Pi\cap\{x:\, x_1<b\}.
\end{equation*}
Let $\g\in\mathds{R}$, $b>0$ be some numbers, $\vr_b$ be the characteristic function of the set $\overline{\Pi_b}$. By  $W_{2,\g,b}^{2,\pm}$ we denote the weighted spaces consisting of the functions  $u\in\Ho^2(\Pi_b^{\pm},\p\Pi_b^{\pm}\cap\p\Pi)$ having a finite norm
\begin{equation*}
\|u\|_{W_{2,\g,b}^{2,\pm}}:=\|e^{\g  x_1}u\|_{\H^2(\Pi_b^\pm)}.
\end{equation*}

The following theorem describes an analytic continuation of the resolvent of the operator $\Op_\ell$; it was proved in \cite{PMA2020}.

\begin{theorem}\label{th2.1}
There exists a small fixed neighbourhood $\Xi$ of the point   $\l_0$ in the complex plane independent of $\ell$ and the form of the coefficients of the operator $\Op_\ell$ as $-a_-^{(0)}+X_\ell^{(1)}<x_1<a_+^{(n)}+X_\ell^{(n)}$ such that in this neighbourhood, the bordered resolvent  $\vr_b(\Op_\ell-\l)^{-1}\vr_b$ admits an analytic continuation from the upper half-plane into the lower one simultaneously for all  $b>0$. The poles of this continuation called resonances can be equivalently defined as values of $\l$, for which there exits  a non-trivial solution  in   $\Holoc^2(\Pi)$ of the boundary value problem
\begin{equation}\label{2.14}
\begin{gathered}
\hat{\Op}_\ell\psi
-\l \psi=0\quad\text{in}\quad\Pi,
\qquad \psi=0\quad\text{on}\quad\p\Pi,
\\
 \psi(x,\l)=\sum\limits_{p\in P_+}
c_p^\pm \psi_p^\pm(x,\l)+\hat{\psi}_\pm(x,\l)\quad\text{as}\quad  \pm x_1>\pm b_\pm.
\end{gathered}
\end{equation}
Here $\hat{\Op}_\ell$ is the differential expression in the right hand side in    (\ref{2.6}),  $c_p^\pm$ are some constants, $\hat{\psi}_+\in W_{2,\g_+,b_+}^{2,+}$ and 
$\hat{\psi}_-\in W_{2,\g_-,b_-}^{2,-}$ are some functions, $b_+:=X_\ell^{(n)}+a$,  $b_-:=X_\ell^{(1)}-a$, and  $\g_+>0$, $\g_-<0$ are some numbers chosen so that for all    $\l\in\Xi$ the strip $\big\{\tau:\, \g_-< \IM\tau < \g_+\big\}$ contains no eigenvalues of the quadratic operator pencils $\Op^{(0)}_{per}(\tau)-\l$ and $\Op^{(n)}_{per}(\tau)-\l$
 except for $\mathfrak{t}_p^\pm(\l)$.
\end{theorem}

\subsection{Behavior at infinity}

For each  $k\in\{0,\ldots,n\}$,  by $N^{(k)}$ we denote the multiplicity of  $\l_0$ considered as an eigenvalue of the operator $\Op^{(k)}$ and $\phi_p^{(k)}=\phi_p^{(k)}(x)$, $p=1,\ldots,N^{(k)}$ are the associated eigenfunctions orthonormalized in  $L_2(\Pi)$.  If  $\l_0$ is not an eigenvalue of the operator  $\Op^{(k)}$ for some $k$, we let $N^{(k)}:=0$, $\phi_p^{(k)}:=0$. It is clear that under our assumptions the point $\l_0$ is a discrete eigenvalue  or a point in the resolvent set of each operator $\Op^{(k)}$, $k\in\{2,\ldots,n-1\}$. We also stress that it can not be an eigenvalue of an infinite multiplicity for the operators $\Op^{(k)}$, $k\in\{1,n\}$ since such situation is excluded by inequalities (\ref{2.13}).

In order to formulate our main results, we need to describe the behavior at infinity of the eigenfunctions $\phi_p^{(k)}$ associated with the eigenvalue $\l_0$. While in  simplest cases
this issue is trivially resolved by an appropriate separation of variables, the situation is more complicated for arbitrary periodically varying coefficients. In order to describe this situation, we employ general results from \cite{NP}; they provide a needed description in terms of the above introduced quadratic operator pencils $\Op^{(k)}_{per}(\tau)-\l_0$. The final description is not complicated but involves some rather bulky technical details, which we have to deal with.

According \cite[Ch. 3, Sect. 4, Prop. 4.4, Stat. 1]{NP}, the eigenvalues of the quadratic  operator pencils $\Op^{(k)}_{per}(\tau)-\l_0$ are isolated, form a countable set and can accumulate at infinity only in an angle $|\IM\tau|>C|\RE\tau|$, $C=const$. It is easy to see that
\begin{equation*}
\big(\Op^{(k)}_{per}(\tau)-\l_0\big)^*=\Op^{(k)}_{per}(\overline{\tau})-\l_0,\qquad \tau\in\mathds{C}.
\end{equation*}
The operator $\Op^{(k)}_{per}(\tau)-\l_0$ is Fredholm, see \cite[Ch. V\!I\!I\!I, Sect. 29.3]{VT}. Hence, the eigenvalues of the operator pencil $\Op^{(k)}_{per}(\tau)-\l_0$ are complex conjugate. The eigenvalues with positive imaginary parts are denoted by $\mathfrak{r}^{(k,+)}_i$, $i=1,2,\ldots$, while ones with negative imaginary parts are denoted by $\mathfrak{r}^{(k,-)}_i=\overline{\mathfrak{r}^{(k,+)}_i}$, $i=1,2,\ldots$ Taking the multiplicities  into account, we arrange these eigenvalues as follows:
%\begin{equation}\l%abel{2.28}
\begin{align*}
&\IM \mathfrak{r}_1^{(k,+)}\leqslant \IM \mathfrak{r}_2^{(k,+)}\leqslant \IM \mathfrak{r}_3^{(k,+)}\leqslant \ldots,
\qquad 
\IM \mathfrak{r}_1^{(k,-)}\geqslant \IM \mathfrak{r}_2^{(k,-)}\geqslant \IM \mathfrak{r}_3^{(k,-)}\geqslant  \ldots
\end{align*}
We denote
$\mr:= \min\limits_{k=1,\ldots,n-1} \IM \mathfrak{r}_1^{(k,+)}$.
For each $k=1,\ldots,n-1$ we consider the eigenvalues of the operator pencil  $\Op^{(k)}_{per}(\tau)-\l_0$ satisfying the condition
\begin{equation}\label{2.28c}
\IM \mathfrak{r}_i^{(k,\pm)}=\pm\mr;
\end{equation}
these eigenvalues are $\mathfrak{r}_i^{(k,\pm)}$ for $i=1,\ldots,J^{(k)}$ for some $J^{(k)}\geqslant 0$. The identity $J^{(k)}=0$ corresponds to the situation, when the
operator pencil  $\Op^{(k)}_{per}(\tau)-\l_0$ possesses no eigenvalues $\mathfrak{r}_i^{(k,\pm)}$   obeying condition  (\ref{2.28c}). We note that thanks to an aforementioned  discreteness of the set $\{\mathfrak{r}_i^{k,\pm}\}_{i\in\mathds{N}}$, there exists $\g\in(\mr,2\mr)$ such that the strip $
\big\{\tau:\; |\IM\tau|\leqslant \g \big\}$ contains no other eigenvalues of $\Op^{(k)}_{per}(\tau)-\l_0$ except for $\mathfrak{r}_i^{(k,\pm)}$, $i=1,\ldots,J^{(k)}$.

For each eigenvalue   $\mathfrak{r}^{(k,\pm)}_{i}$ there exists an associated Jordan chain  $\Phi^{(k,\pm)}_{is}$, $s=0,\ldots,\vk^{(k)}_{i}-1$, $\vk^{(k)}_{i}\geqslant 1$ is the length of the chain; the lengths of the chains associated with both eigenvalues $\mathfrak{r}^{(k,-)}_{i}$  and $\mathfrak{r}^{(k,+)}_{i}$ are same, see \cite[Ch. 1, Sect. 2, Prop. 2.2]{NP}.
 The functions $\Phi^{(k,\pm)}_{i0}$ are the eigenfunctions associated with the eigenvalue  $\mathfrak{r}^{(k,\pm)}_{i}$, while the adjoint functions solve the equations
%\begin{equation}\l%abel{2.24}
\begin{align*}
\big(\Op^{(k)}_{per}(\mathfrak{r}^{(k,\pm)}_{i}) -\l_0\big)\Phi^{(k,\pm)}_{i1}&+ \frac{\p\Op^{(k)}_{per}}{\p\tau}(\mathfrak{r}^{(k,\pm)}_{i})
\Phi^{(k,\pm)}_{i0}=0,
\\
\big(\Op^{(k)}_{per}(\mathfrak{r}^{(k,\pm)}_{i}) -\l_0\big)\Phi^{(k,\pm)}_{is} &+\frac{\p\Op^{(k)}_{per}}{\p\tau}(\mathfrak{r}^{(k,\pm)}_{i})
\Phi^{(k,\pm)}_{i\,s-1}
+ \frac{1}{2}\frac{\p^2\Op^{(k)}_{per}}{\p\tau^2} (\mathfrak{r}^{(k,\pm)}_{i})
\Phi^{(k,\pm)}_{i\,s-2}=0,
\end{align*}
%\end{equation}
as $s\geqslant 2$, see \cite[Ch. 1, Sect. 2, Subsect. 2]{NP}.

We continue the functions  $\Phi^{(k,\pm)}_{is}$  $T^{(k)}$-periodically in  $x_1$ keeping the same notations. We denote
\begin{equation}\label{2.25}
\vp^{(k,\pm)}_{is}(x):=e^{\pm\iu\mathfrak{r}^{(k,\pm)}_{i} x_1}
\tilde{\vp}^{(k,\pm)}_{is}(x),\qquad \tilde{\vp}^{(k,\pm)}_{is}(x):=
\sum\limits_{p=0}^{s}
\frac{(\iu x_1)^p}{p!}\Phi^{(k,\pm)}_{i\,s-p}(x),
\end{equation}
where  $s=0,\ldots, \vk^{(k)}_{i}-1$.
According \cite[Ch. 5, Sect. 1.3, Thm. 1.4]{NP}, the introduced functions provide the leading terms in the asymptotics at infinity for the eigenfunctions $\phi_p^{(k)}$ of the operators $\Op^{(k)}$:
\begin{align}
&\phi_p^{(k)}(x)=\tilde{\phi}_p^{(k,\pm)}(x)  + \hat{\phi}_p^{(k,\pm)}(x), \qquad \pm x_1>a,
\label{2.22}
\\
& \tilde{\phi}_p^{(k,+)}:=\sum\limits_{i=1}^{J^{(k)}}\sum\limits_{s=0}^{\vk^{(k)}_{i}-1} \a_{pis}^{(k,+)} \vp^{(k,+)}_{is},\qquad k=1,\ldots,n-1,\nonumber
\\
&\tilde{\phi}_{p}^{(k,-)}:= \sum\limits_{i=1}^{J^{(k-1)}}\sum\limits_{s=0}^{\vk^{(k-1)}_{i}-1} \a_{pis}^{(k-1,-)} \vp^{(k-1,-)}_{is},\qquad k=2,\ldots,n,\nonumber
\end{align}
where $\a_{pis}^{(k,\pm)}$ are some numbers and
$\hat{\phi}_p^{(k,\pm)}\in \mathring{W}_{2,\pm\g,\pm a}^{2,\pm}$  are some functions.

We also observe that under the shift  $x_1\mapsto x_1+T$, where $T$ is a multiple of the period $T^{(k)}$, the function $\tilde{\phi}_p^{(k,+)}$ is transformed as follows:
\begin{equation}\label{2.26}
\begin{aligned}
&\tilde{\phi}_{p}^{(k,+)}(x_1+T,x')=\sum\limits_{i=1}^{J^{(k)}}
e^{\iu\mathfrak{r}^{(k,+)}_{i} T}
\sum\limits_{s=0}^{\vk^{(k)}_{i}-1} \b_{pis}^{(k,+)}(T) \vp^{(k,+)}_{is}(x), 
\\
&
\b_{pis}^{(k,+)}(T):=\sum\limits_{m=0}^{\vk_i^{(k)}-s-1} \frac{\a_{pi\,m+s}^{(k,+)}}{m!}(\iu T)^m.
\end{aligned}
\end{equation}
Similarly, if  $T$ is a multiple of $T^{(k-1)}$, then
\begin{equation}\label{2.27}
\begin{aligned}
&\tilde{\phi}_{p}^{(k,-)}(x_1-T,x')=\sum\limits_{i=1}^{J^{(k-1)}}
e^{-\iu\mathfrak{r}^{(k-1,-)}_{i} T}
\sum\limits_{s=0}^{\vk^{(k-1)}_{i}-1} \b_{pis}^{(k-1,-)}(T) \vp^{(k-1,-)}_{is}(x),
\\
&\b_{pis}^{(k-1,-)}(T):=\sum\limits_{m=0}^{\vk_i^{(k-1)}-s-1} \frac{\a_{pi\,m+s}^{(k-1,-)}}{m!}(-\iu T)^m.
\end{aligned}
\end{equation}

We denote
\begin{gather*}
  \vk:=\max\limits_{\substack{k=1,\ldots,n-1
  \\
  i=1,\ldots,J^{(k)}}} \vk_i^{(k)},\qquad  \eta(\ell):=\|\ell\|^\vk  e^{-\mr \la\ell\ra},
  \\
\la\ell\ra:=\min\limits_{k=1,\ldots,n-1}\big\{|X_\ell^{(k+1)}-X_\ell^{(k)}|\big\},\qquad
\|\ell\|:=\max\limits_{k=1,\ldots,n-1} \big\{|X_\ell^{(k+1)}-X_\ell^{(k)}|\big\}.
\end{gather*}

\subsection{Main result}

We let
\begin{equation}\label{2.30}
N:= \sum\limits_{k=1}^{n}  N^{(k)}.
\end{equation}
and consider  an auxiliary block matrix $\sAo=\sAo(\ell)$ of size $N\times N$:
\begin{equation*}%\l%abel{2.29}
\sAo=\begin{pmatrix}
0 & \sAo_{12} & 0 & 0 & \ldots & 0 & 0 & 0 & 0
\\
\sAo_{21} & 0 & \sAo_{23} & 0 & \ldots & 0 & 0 & 0 & 0
\\
0 & \sAo_{32} & 0 & \sAo_{34} & \ldots & 0 & 0 & 0 & 0
\\
\vdots & \vdots & \vdots & \vdots& \ddots & \vdots & \vdots & \vdots & \vdots
\\
0 & 0 & 0 & 0 & \ldots & 0 & \sAo_{n-1n-2} & 0 & \sAo_{n-1n}
\\
0 & 0 & 0 & 0 & \ldots & 0 & 0 & \sAo_{nn-1} & 0
\end{pmatrix}.
\end{equation*}
The block located at the intersection of $r$th row and $k$th column is of the size $N^{(r)}\times N^{(k)}$; the  matrices $\sAo^{(0)}_{k\pm1\,k}$ read as
\begin{align}\label{5.55}
&\sAo_{k+1\,k}(\ell):=
\begin{pmatrix}
\cAo^{(k,+)}_{jp}(\ell)
\end{pmatrix}_{p=1,\ldots,N^{(k)}}^{j=1,\ldots,N^{(k+1)}},\quad
\sAo_{k\,k+1}(\ell):=
\begin{pmatrix}
\cAo^{(k,-)}_{jp}(\ell)
\end{pmatrix}_{p=1,\ldots,N^{(k+1)}}^{j=1,\ldots,N^{(k)}},
\\
&
\cAo^{(k,+)}_{jp}(\ell):= \sum\limits_{i=1}^{J^{(k)}}
e^{\iu\mathfrak{r}^{(k,+)}_{i} (X_\ell^{(k+1)}-X_\ell^{(k)})}
\sum\limits_{q=1}^{J^{(k)}}\sum\limits_{s=0}^{\vk^{(k)}_{i}-1}\sum\limits_{t=0}^{\vk_{q}^{(k)}-1} \overline{\a_{jqt}^{(k,-)}} K_{isqt}^{(k)}\b_{pis}^{(k,+)}(X_\ell^{(k+1)}-X_\ell^{(k)}),
\nonumber
%\l%abel{5.56}
\\
&
\cAo^{(k,-)}_{jp}(\ell):= \sum\limits_{i=1}^{J^{(k)}}
e^{-\iu\mathfrak{r}^{(k,-)}_{i} (X_\ell^{(k+1)}-X_\ell^{(k)})}
\sum\limits_{q=1}^{J^{(k)}} \sum\limits_{s=0}^{\vk^{(k)}_{i}-1} \sum\limits_{t=0}^{\vk_{q}^{(k)}-1}\overline{\a_{jqt}^{(k,+)}} K_{qtis}^{(k)} \b_{pis}^{(k,-)}(X_\ell^{(k+1)}-X_\ell^{(k)}).\nonumber
%\l%abel{5.57}
\end{align}
The expressions $(\cdot)^{j=1,\ldots,N^{(r)}}_{p=1,\ldots,N^{(k)}}$ denote the matrices, where the superscript  $j$ counts the index of the row, while the subscript  $p$ does the index of the column. The constants $K^{(k)}_{isqt}$ are defined by the formulae:
\begin{equation}\label{5.45}
K^{(k)}_{isqt}=0\quad\text{as}\quad \mr_{i}^{(k,+)}\ne \mr_{q}^{(k,+)},
\end{equation}
and
\begin{equation}
\label{5.46}
\begin{aligned}
K^{(k)}_{isqt}=\int\limits_{\om}  &\overline{\Phi_{qt}^{(k,-)}} \left(
\frac{\p \Phi_{is}^{(k,+)}}{\p\nu^{(k)}}- \iu A_{11,per}^{(k)} \Phi_{i\,s-1}^{(k,+)} \right)
- \Phi_{is}^{(k,+)}    \left(
\frac{\p\overline{\Phi_{qt}^{(k,-)}}}{\p\nu^{(k)}} +\iu A_{11,per}^{(k)}\overline{\Phi_{q\,t-1}^{(k,-)}} \right)\Bigg|_{x_1=T^{(k)}}\di x'
\end{aligned}
\end{equation}
as $\mr_{i}^{(k,+)}= \mr_{q}^{(k,+)}$; as $s=0$ or $t=0$, in  (\ref{5.46}) we let $\Phi_{i\,-1}^{(k,+)}=0$, $\Phi_{q\,-1}^{(k,-)}=0$.

By $\L_j=\L_j(\ell)$, $j=1,\ldots,N$, we denote the eigenvalues of the matrix $\sAo(\ell)$ taken counting their algebraic multiplities. We partition the set of these eigenvalues into disjoint subgroups $\{\L_j(\ell)\}_{j\in L_p}$, $\bigcup\limits_{p} L_p=\{1,\ldots,N\}$, so that  eigenvalues $\L_i$ and $\L_j$ in the same group satisfy the bound
\begin{equation}\label{2.31}
|\L_i(\ell)-\L_j(\ell)|\leqslant C\eta^{1-\frac{1}{N}}(\ell)e^{-\frac{\g}{N}\la \ell\ra},
\end{equation}
while for the eigenvalues in different groups the identity holds:
\begin{equation}\label{2.32}
|\L_i(\ell)-\L_j(\ell)|\geqslant\mu(\ell)\eta^{1-\frac{1}{N}}(\ell)e^{-\frac{\g}{N}\la \ell\ra}.
\end{equation}
The symbol $C$ in (\ref{2.31}) stands for some fixed constant independent of $i$, $j$ and $\ell$, while the symbol $\mu(\ell)$ in (\ref{2.32}) denotes a positive function such that $\lim\limits_{\ell\to\infty}\mu(\ell)=+\infty$.

\begin{theorem}\label{th2.2}
As $\ell$ is large enough, there exists a fixed independent of $\ell$  neighbourhood $\Xi$ of the point $\l_0$ in the complex plane, in which the  operator $\Op_\ell$ possesses at most $N$ resonances. For each group $L_p$ there exists at least one resonance of the operator $\Op_\ell$ with the asymptotics
\begin{equation}\label{2.33}
\l_\ell=\l_0+\L_i(\ell)+O\big(\eta^{1-\frac{1}{N}}(\ell)e^{-\frac{\g}{N}\la \ell\ra}\big),\qquad \ell\to\infty,
\end{equation}
where $\L_i\in L_p$. The operator $\Op_\ell$ has at most $\# L_p$ resonances with asymptotics (\ref{2.33}), where $\# L_p$ is the number of the elements in the set $L_p$. All eigenvalues $\L_i$, $i=1,\ldots,N$, are of order $O(\eta(\ell))$ as $\ell\to\infty$ and this estimate is order sharp.
\end{theorem}

Let us discuss briefly the main aspects of our problem and of the result. The problem is rather general since the operator $\Op_\ell$ is multi-dimensional and is ``glued'' from  a fixed but arbitrary number of distant perturbations described by the operators $\Op^{(k)}$. As $x_1\in(X_\ell^{(k-1)},X_\ell^{(k)})$, the coefficients of this operator are  periodic functions $A_{\natural,per}^{(k)}$, which are not supposed to be same for different $k$.
Our model is quite rich and includes, in particular,  various situations like truncated periodic potentials similar to one considered in \cite{Kl}, multiple potential and magnetic wells separated by large distances, distant perturbations of metric, etc. Our main assumption is formulated in the end of Subsection~\ref{ss:pr} in terms of the spectra of the operators $\Op^{(k)}$; additional assumption  (\ref{2.13}) is rather technical and is made for simplicity. Under these assumptions, our main result states that the operator $\Op_\ell$ can have only finitely many resonances in the vicinity of $\l_0$ and all these resonances converge to $\l_0$ and have asymptotics (\ref{2.33}). The total number of different resonances is at most  $N$ but can be less. Of course, if all $\L_i$ are distinct, we have exactly $N$ different resonances. But in the general situation  some resonances of the operator $\Op_\ell$ can coincide. If this is the case, the total multiplicity of all perturbed resonances as the number of associated linear independent solutions of problem (\ref{2.14}) is again at most $N$ but can be less  since instead of some such generalized eigenfunctions, adjoint functions can exist. We also stress that if $\l_0$ is not an eigenvalue of all operators $\Op^{(k)}$ as $k=1,\ldots,n$, then our result implies that the operator $\Op_\ell$ has no resonances in the vicinity of $\l_0$.

As we have said above, we assume that  $\l_0\notin\essspec(\Op^{(k)})$, $k=2,\ldots,n-1$,  and
thanks to Lemma~\ref{lm2.1}, this is equivalent to $\l_0\notin\essspec(\Op^{(k)}_{per})$, $k=2,\ldots,n-1$. This assumption is very essential since it   ensures that the operator $\Op_\ell$ has only finitely many resonances in the vicinity of the point $\l_0$. If $\l_0$ is in the essential spectrum of at least one of the operators $\Op^{(k)}_{per}$, $k=2,\ldots,n-1$, then the situation changes dramatically. As we have said in Introduction, particular examples considered in \cite{Bar}, \cite{JPA19}, \cite{AML20}, \cite{HoMeb} show that in this case there can be a growing number of closely spaced resonances in the vicinity of the point $\l_0$. We do expect that such situation \textit{always occurs} in our general model once
$\l_0$ is in the essential spectrum of at least one of the operators $\Op^{(k)}_{per}$, $k=1,\ldots,n-1$. This is an interesting problem, which we shall study in our next work. Right now we stress that in view of our conjecture, the situation studied in the present work is the only case, when finitely many resonances emerge from the point $\l_0$.

It should be also said that the leading terms $\L_i$ in asymptotics (\ref{2.33}) are exponentially small since $\eta$ is exponentially small. Hence, all considered resonances of the operator $\Op_\ell$ are exponentially close to $\l_0$. Although we provide only leading terms in asymptotics (\ref{2.33}), our technique allows one to find next terms in this asymptotics. The point is that all  resonances of the operator $\Op_\ell$ are zeroes of certain equation, see (\ref{5.25}). In fact, in the proof of Theorem~\ref{th2.2}, we find a leading term for the matrix $\sA$ in this equation, which turns out to be $\sAo$ and its eigenvalues determines the leading terms in (\ref{2.33}). It also possible to find a more detailed asymptotics for the matrix $\sA$ and this will determine next terms in asymptotics (\ref{2.33});  these are rather bulky and technical but simple calculations. We also mention that at least formally, next terms in asymptotics (\ref{2.33}) can be constructed via the scheme proposed in \cite[Sect. 5]{GolPMA}; however, an issue of justification of this scheme for the resonances seems to be more complicated than for the isolated eigenvalues studied in  \cite{GolPMA}.

\section{Auxiliary problems}\label{s:Res}

In the present section we collect some preliminaries on the behavior of solutions to some auxiliary problems, which  will be then employed in the proof of our main result.
 We introduce two auxiliary spaces
\begin{equation*}
\mW:=W_2^2(\Om_-)\oplus W_2^2(\Om_+),\qquad  \Om_\pm:=\{x\in\Pi:\, 0<\pm x_1<a\},\qquad \mV:=L_2(\Pi_a)\oplus\H^\frac{3}{2}(\om)\oplus\H^\frac{1}{2}(\om),
\end{equation*}
and  an operator family $\cF^{(k)}(\l):\,\mW\to \mV$ acting by the following rule. Given $g\in \mW$, the components of the vector  $\cF^{(k)}(\l) g=(F,f_0,f_1)\in \mV$ are defined as
\begin{equation}\label{4.22}
F:=\left\{
\begin{aligned}
(\hat{\Op}^{(k)}&-\l)g\Big|_{\Om_\pm}\quad&&\text{on}\quad\Om_\pm,
\\
&0&&\text{otherwise},
\end{aligned}\right.
\qquad
f_0:=[g]_0,\qquad f_1:=\left[\frac{\p g}{\p\nu^{(k)}}\right]_0.
\end{equation}
We consider boundary value problems:
\begin{align}\label{3.26}
&\big(\hat{\Op}^{(k)}-\l\big)u^{(k)}=F \quad\text{in}\quad\Pi\setminus\om_0,
\qquad u=0\quad\text{on}\quad\p\Pi,
\qquad [u^{(k)}]_0=f_0,\qquad \left[\frac{\p u^{(k)}}{\p\nu^{(k)}}\right]_0=f_1,
\\
&
\begin{aligned}
&u^{(k)}\in\H^2(\Pi_0^+)\quad\text{as}\quad k=1,\ldots,n-1,\qquad u^{(n)}\in\Hloc^2(\Pi_0^+),
\\
& u^{(k)}\in\H^2(\Pi_0^-)\quad\text{as}\quad k=2,\ldots,n,\qquad u^{(1)}\in\Hloc^2(\Pi_0^-),
\end{aligned}
\label{3.27}
\\
&u^{(1)}(x,\l)=\sum\limits_{p\in P_-}
\big(c_{p,-}(\l) g \big) \psi_p^-(x,\l)+\hat{u}_-^{(1)}(x,\l)\qquad \text{as}\quad x_1<-a, \qquad \hat{u}_-^{(1)}(x,\l)\in W_{2,\tilde{\mr},a}^{2,-}, \label{5.7d}
\\
&
u^{(n)}(x,\l)=\sum\limits_{p\in P_+}\big(c_{p,+}(\l)g\big) \psi_p^+(x,\l)+\hat{u}_+^{(1)}(x,\l)\qquad \text{as}\quad x_1>a,\hphantom{-}\qquad \hat{u}_+^{(n)}(x,\l)\in W_{2,\tilde{\mr},a}^{2,+},
\label{5.7f}
\end{align}
where $[v]_0:=v\big|_{x_1=+0}-v\big|_{x_1=-0}$ for an arbitrary $v\in\H^1\big(\Pi_a\setminus\om_0)\big)$,
\begin{equation*}
 \om_0:=\{0\}\times\om, \qquad
\frac{\p\quad\;}{\p\nu^{(k)}}:=-\sum\limits_{i=1}^{d}A_{1i,per}^{(k)}\frac{\p\ }{\p x_i} + \iu A_{1,per}^{(k)},
\end{equation*}
$c_{p,\pm}(\l)$ are some linear bounded functionals on $\mV$
and $0<\tilde{\mr}<\mr$ is some fixed number. By $\cT_\pm:\, L_2(\Pi)\to L_2(\Om_\pm)$ we denote the operator of restriction on $\Om_\pm$, that is, $\cT_\pm u:=u\big|_{\Om_\pm}$.

A following lemma describes the behavior of solutions of problems (\ref{3.26}), (\ref{3.27}) as $k\in\{2,\ldots,n-1\}$.

\begin{lemma}\label{lm5.4}
Let $\l_0$  be a discrete eigenvalue of the operator $\Op^{(k)}$, $k\in\{2,\ldots,n-1\}$. There exists a small neighbourhood $\Xi_k$ of the point $\l_0$ such that for $\l\in\Xi_k\setminus\{\l_0\}$ problems  (\ref{3.26}), (\ref{3.27}) are uniquely solvable and the following statements hold. The operators
\begin{equation}\label{3.15}
\cU^{(k)}=\cU^{(k)}(\l):\, \mW\to\H^2(\Pi^-_0)\oplus\H^2(\Pi^+_0)
\end{equation}
 mapping $g$ into the solution of these problems can be represented as
\begin{gather}\label{6.4a}
\cU^{(k)}=\sum\limits_{p=1}^{N^{(k)}}
\frac{\phi_{p}^{(k)}}
{\l_0-\l} \cP^{(k)}_p(\l)+\cR^{(k)}(\l),
\\
\cP_p^{(k)}(\l)g:=(F,\phi_p^{(k)})_{L_2(\Pi_a)} + (f_1,\phi_p^{(k)})_{L_2(\om_0)} - \bigg(f_0,\frac{\p\phi_p^{(k)}}{\p\nu^{(k)}}\bigg)_{L_2(\om_0)},\label{3.28}
\end{gather}
where $\mathcal{R}^{(k)}(\l)$ are the reduced resolvents being bounded operators from $\mW$ into
$\H^2(\Pi^-_0)\oplus\H^2(\Pi^+_0)$ holomorphic in  $\l\in \Xi_k$.  The reduced resolvents act into the orthogonal complement in  $L_2(\Pi)$ to the eigenspace spanned over the eigenfunctions $\phi_k^{(p)}$, $p=1,\ldots,N^{(k)}$, and solve boundary value problem (\ref{3.26}), (\ref{3.27}) with $F$ replaced by $F-\sum\limits_{p=1}^{N^{(k)}} \phi_{p}^{(k)}
\cP^{(k)}_p g$.

For $\l\in\Xi_k$ and positive sufficiently large $X$, the estimates hold:
\begin{equation}\label{3.3}
\|\cT_\pm\cS(\pm X)\cU^{(k)}(\l)\|_{\mW\to\H^2(\Om_\pm)} \leqslant C  e^{-\frac{3\mr}{4}X},
\end{equation}
where $c$ and $C$ are some constant independent of $=w$, $X$ and $\l$. The operators $\cT\cS(\pm X)\cU^{(k)}(\l)$ are holomorphic in $\l\in\Xi_k$.
\end{lemma}

\begin{proof}
We seek a solution to problem  (\ref{3.26}), (\ref{3.27}) as $u^{(k)}=v+\chi g$, where $\xi_1=\xi_1(x_1)$ is an infinitely differentiable cut-off function equalling to one as $|x_1|<\frac{a}{3}$ and vanishing as $|x_1|>\frac{2a}{3}$. Then for a new unknown function $v$ we obtain the equation
\begin{equation}\label{3.16}
(\Op^{(k)}-\l)v=\tilde{F},\qquad \tilde{F}:=(\Op^{(k)}-\l)(1-\xi_1)g.
\end{equation}
According formula (3.21) and other results in \cite[Ch. 5, Sect 3.5]{Kato}, this equation is uniquely solvable for $\l\in\Xi_k\setminus\{0\}$, where $\Xi_k$ is some sufficiently small fixed neighbourhood of the point $\l_0$ in the complex plane. The solution is meromorphic in $\l$ having a simple pole at $\l_0$:
\begin{equation}\label{3.8}
v=(\Op^{(k)}-\l)^{-1}=\sum\limits_{p=1}^{N^{(k)}}
\frac{(\tilde{F},\phi_p^{(k)})_{L_2(\Pi_a)}}
{\l_0-\l}\phi_{p}^{(k)} +\tilde{\cR}(\l)\tilde{F},
\end{equation}
where $\tilde{\cR}(\l)$ is a reduced resolvent; this is a bounded operator from $L_2(\Pi_a)$ into $\Ho^2(\Pi)$ holomorphic in $\l\in\Xi_k$. The reduced resolvent acts in the orthogonal complement in $L_2(\Pi_a)$ to the eigenspace spanned over the eigenfunctions $\phi_k^{(p)}$, $p=1,\ldots,N^{(k)}$. It solves the equation
\begin{equation}\label{3.9}
(\Op^{(k)}-\l)\tilde{\cR}(\l)\tilde{F}=\tilde{F}- \sum\limits_{p=1}^{N^{(k)}}
 (\tilde{F},\phi_p^{(k)})_{L_2(\Pi_a)} \phi_{p}^{(k)}.
\end{equation}

Let us find the scalar products $(\tilde{F},\phi_p^{(k)})_{L_2(\Pi_a)}$. Employing the definition of the function $\tilde{F}$ and the eigenvalue equation for the functions $\phi_p^{(k)}$, we integrate by parts:
\begin{equation}
\begin{aligned}
(\tilde{F},\phi_p^{(k)})_{L_2(\Pi_a)} = & \int\limits_{\Pi_a}\overline{\phi_p^{(k)}}(\tilde{\Op}^{(k)}-\l)(1-\xi_1)g\di x
\\
=& \int\limits_{\Pi_a}\overline{\phi_p^{(k)}}F\di x -\int\limits_{\Pi_a\setminus\om_0}\overline{\phi_p^{(k)}} (\tilde{\Op}^{(k)}-\l_0)\xi_1 g\di x + (\l-\l_0)\int\limits_{\Pi_a}\overline{\phi_p^{(k)}} \xi_1 g\di x
\\
=&\cP_p^{(k)}(\l)g + (\l-\l_0)(\xi_1 g,\phi_p^{(k)})_{L_2(\Pi_a)}.
\end{aligned}
\end{equation}
We substitute this formula into (\ref{3.8}) and then for $u^{(k)}$ we obtain
representation (\ref{6.4a}) with
\begin{equation}\label{3.13}
\cR(\l)g=\tilde{\cR}(\l)\tilde{F}+ \xi_1 g-\sum\limits_{p=1}^{N^{(k)}}
(\xi_1 g,\phi_p^{(k)})_{L_2(\Pi_a)}\phi_{p}^{(k)}.
\end{equation}
Hence, in view of the aforementioned properties of the operator $\tilde{\cR}(\l)$, the operator $\cR(\l)$ is bounded from $\mW$ into $\H^2(\Pi^-_0)\oplus\H^2(\Pi^+_0)$  and is holomorphic in  $\l\in \Xi_k$. A desired boundary value problem for $\cR(\l)g$ follow identity (\ref{3.13})
and problem (\ref{3.9}).

We proceed to studying the operator  $\cT\cS(X)\cU^{(k)}(\l)$; the   case of the operator  $\cT\cS(-X)\cU^{(k)}(\l)$ can be studied in the same way. Let $\xi_2=\xi_2(x_1)$ be an infinitely differentiable cut-off function equalling to one as $x_1>a+2$ and vanishing as $x_1<a+1$. As $x_1>a$, the coefficients of the operator $\Op^{(k)}$ are periodic and hence, in view of the problem for $\cR(\l)g$, the function $w=\xi_2 \cR(\l)g$ solves the equation
\begin{equation}\label{3.14}
(\Op^{(k)}_{per}-\l)w=\breve{F}-\sum\limits_{p=1}^{N^{(k)}}\xi_2 \phi_{p}^{(k)}
\cP^{(k)}_p g,\qquad
\end{equation}
where $\breve{F}=\breve{F}(x,\l)$ is some compactly supported function in $L_2(\Pi)$ being non-zero only as $x_1\in(a+1,a+2)$; the mapping $g\mapsto\breve{F}$ is linear bounded operator from $\mW$ into $L_2(\Pi)$ holomorphic in $\l\in\Xi_k$.

Since
\begin{equation*}
(\Op^{(k)}_{per}(\tau)-\l)^{-1}= \big(\cI+(\l-\l_0)(\Op^{(k)}_{per}(\tau)-\l_0)^{-1}\big)^{-1}
\big(\Op^{(k)}_{per}(\tau)-\l_0\big)^{-1},
\end{equation*}
where $\cI$ is the identity mapping, and the inverse operator $\big(\Op^{(k)}_{per}(\tau)-\l_0\big)^{-1}$ has poles at the eigenvalues $\mr_i^{(k)}$ of order at most $\vk$, \cite[Ch. 1, Sect. 2, Thm. 2.8]{NP}, it is easy to confirm that the eigenvalues of the operator pencil $\Op^{(k)}_{per}(\tau)-\l$ are located in $C|\l-\l_0|^{\frac{1}{\vk}}$-neighbourhoods of the eigenvalues $\mr_i^{(k,\pm)}$, where $C$ is some fixed constant. Hence, lessening the neighbourhood $\Xi_k$ if it is needed, we conclude that the operator pencil $\Op^{(k)}_{per}(\tau)-\l$ has no eigenvalues in the strip $|\IM\tau|<\frac{7\mr}{8}$ as $\l\in\overline{\Xi_k}$. Then we can apply  Theorem~1.4 in \cite[Ch. 5, Sect.  1.3]{NP} to problem (\ref{3.14}) and we conclude that $w\in$ and the mapping $g\mapsto \cT\cS(X)w$ is a linear bounded operator from $\mW$ into $\H^2(\Pi_a)$
holomorphic in $\l\in\Xi_k$ and satisfying estimate (\ref{3.3}). Since $u=w$ as $x_1>a+2$, this implies all desired properties of the operator $\cT\cS(X)\cU^{(k)}$. The proof is complete.
\end{proof}

\begin{remark}\label{rm4.1}
We observe that if $\l_0$ is not in the spectrum of the operator $\Op^{(k)}$, then the statement of Lemma~\ref{lm5.4} also holds with $N^{(k)}=0$ and $(\mathcal{H}^{(k)}-\lambda)^{-1}=\mathcal{R}^{(k)}(\lambda)$.
\end{remark}

We recall that by our assumptions, the point $\l_0$ can be an eigenvalue of the operators $\Op^{(k)}$, $k\in\{1,n\}$, of multiplicity $N^{(k)}$, embedded into their essential spectra; the identity $N^{(k)}=0$ corresponds to the case, when $\l_0$ is not an eigenvalue of  $\Op^{(k)}$.  The next lemma  describes the behavior of solutions of problems (\ref{3.26}), (\ref{3.27}), (\ref{5.7d}), (\ref{5.7f}) as $k\in\{1,n\}$.

\begin{lemma}\label{lm5.3b}
There exists small fixed neighbourhoods $\Xi_k$, $k\in\{1,n\}$, such that problems (\ref{3.26}), (\ref{3.27}), (\ref{5.7d}), (\ref{5.7f}) are uniquely solvable respectively as $\l\in\Xi_1\setminus\{\l_0\}$ and $\l\in\Xi_n\setminus\{\l_0\}$ and the following statements hold. The operators
(\ref{3.15}) satisfy representation (\ref{6.4a}) with functionals (\ref{3.28}),
where $\mathcal{R}^{(k)}(\l)$ are the reduced resolvents being bounded operators from $\mW$ into
$\Hloc^2(\Pi^-_0)\oplus\H^2(\Pi^+_0)$ for $k=1$ and into $\H^2(\Pi^-_0)\oplus\Hloc^2(\Pi^+_0)$
for $k=n$; these operators are  holomorphic in  $\l\in\Xi_k$. The reduced resolvents satisfy the
orthogonality conditions
\begin{equation*}%\l%abel{5.18}
\int\limits_{\Pi}\overline{\phi_p^{(k)}} \cR^{(k)}(\l) g\di x=0,\qquad p=1,\ldots,N^{(k)}.
\end{equation*}
The functions $\cR^{(k)}(\l) g$ solve  problems  (\ref{3.26}), (\ref{3.27}), (\ref{5.7d}), (\ref{5.7f}) with $F$ replaced by $F-\sum\limits_{p=1}^{N^{(k)}} \phi_{p}^{(k)}
\cP^{(k)}_p g$.

For $\l\in\Xi_k$ and positive sufficiently large $X$, the estimates hold:
\begin{equation}\label{3.29}
\|\cT\cS(-X)\cR^{(1)}(\l)\|_{\mW\to\H^2(\Pi_a)} \leqslant C  e^{-\frac{3\mr}{4}X},\qquad \|\cT\cS(X)\cR^{(n)}(\l)\|_{\mW\to\H^2(\Pi_a)} \leqslant C  e^{-\frac{3\mr}{4}X},
\end{equation}
where $c$ and $C$ are some constant independent of $=w$, $X$ and $\l$ and constant $c$ satisfies the second inequality in (\ref{3.3}). The operators $\cT\cS(-X)\cU^{(1)}(\l)$ and $\cT\cS(X)\cU^{(n)}(\l)$ are holomorphic in $\l\in\Xi_k$, $k\in\{1,n\}$.
\end{lemma}

\begin{proof}
We prove the lemma only for $k=1$; the case $k=n$ can be treated in the same way. 
As in the proof of Lemma~\ref{lm5.4}, we seek a solution as $u^{(1)}=\xi_2 g + v$ and for a new unknown function $v$ we obtain problem (\ref{3.26}), (\ref{3.27}), (\ref{5.7d}) with $F$ replaced by $\tilde{F}$ defined in (\ref{3.16}). According \cite[Thm. 1.3]{PMA2020}, the obtained problem for $v$ is uniquely solvable and the solution satisfies representation (\ref{3.13}) for $v$ with the only difference that now the operator $\tilde{\cR}(\l)$ acts into $\Hloc^2(\Pi^-_0)\oplus\H^2(\Pi^+_0)$ and is holomorphic in $\l\in\Xi_1$ in this sense. The rest of the proof reproduces that  of  Lemma~\ref{lm5.4}.
\end{proof}

\section{Resonances as  zeroes of a holomorphic function}

In the present section we begin the proof of our main result, Theorem~\ref{th2.2}. The study of the resonances of the operator  $\Op_\ell$ emerging in the vicinity of the point $\l_0$  is reduced first to a special operator equation and then we show that these resonances coincide with the zeroes of certain holomorphic function.  Our scheme is based on approach suggested in work
\cite{BEG}, but there are serious modification due to a much more general formulation of the problem.

Throughout the proof we  fix a small neighbourhood $\Xi$ of the point $\l_0$ such that $\Xi\subset\Xi_k$ for all $k\in\{0,n\}$ and $\Xi$ is a subset of the neighbourhood mentioned in Theorem~\ref{th2.1}. In what follows we assume that $\l\in\Xi$ and if needed, this neighbourhood will be lessen without saying this explicitly.

\subsection{Structure of non-trivial solutions}

In this subsection we show that all non-trivial solutions of problem (\ref{2.14}) associated with resonances $\l\in\Xi$ satisfy certain representation, which will serve as a base for the proof of Theorem~\ref{th2.2}.

We choose arbitrary $g_k\in\mW$, $k=1,\ldots,n$, assuming that $g_1\equiv 0$ on $\Om_-$ and $g_n\equiv 0$ on $\Om_+$. By $u_k=u_k(x,\l)$ we denote the solutions of problems (\ref{3.26}), (\ref{3.27}), (\ref{5.7d}), (\ref{5.7f}) with $g=g_k$. Let  $\mathcal{S}(\cdot)$ be a shift operator in $L_2(\Pi)$ acting by the rule $\mathcal{S}\big(\cdot\big)\big(u(x)\big):=u(x_1-\cdot,x')$.

The main aim of the present subsection is to prove the following lemma.

\begin{lemma}\label{lm6.1} For $\ell$ large enough
each non-trivial solution of problem (\ref{2.14}) associated with some $\l\in\Xi$ can be represented as

\begin{align}\label{6.5}
\psi(x,\ell)=
\left\{
\begin{aligned}
& u_1(x_1-X_\ell^{(1)},x') &&\text{as}\quad x_1<X_\ell^{(1)},
\\
&
 \begin{aligned}
 &u_{k-1}(x_1-X_\ell^{(k-1)},x')
 \\
 &\hphantom{u}+ u_k(x_1-X_\ell^{(k)},x')
 \end{aligned}&&\text{as}\quad X_\ell^{(k-1)}<x_1<X_\ell^{(k)},\quad k=2,\ldots,n,
\\
&u_n(x_1-X_\ell^{(n)},x') &&\text{as}\quad x_1>X_\ell^{(n)}.
\end{aligned}
\right.
\end{align}
for some $g_k\in\mW$.
\end{lemma}

The rest of the subsection is devoted to the proof of this lemma. Let $\psi$ be a non-trivial solution of problem (\ref{2.14}) associated with some $\l\in\Xi$ and assume that $\ell$ is large enough. In view of formula (\ref{6.5}), the restriction of the function  $u_1$ on $\Pi_0^-$ and that of $u_n$ on $\Pi_0^+$ are recovered immediately:
\begin{equation}\label{4.23}
u_1(x):=\psi(x_1-X_\ell^{(1)},x'),\quad x_1<0,\qquad u_n(x):=\psi(x_1-X_\ell^{(n)},x'),\quad x_1>0.
\end{equation}
We are going to show how to recover $u_1(x)$ as $x_1>0$ and $u_2(x)$ as $x_2<0$; the procedure for other functions is similar.

We choose some $f_\pm\in\H^2(\Om_\pm)$ and consider two boundary value problems in $\H^2(\Pi_0^-)\oplus\H^2(\Pi_0^+)$:
\begin{equation}\label{4.24}
\begin{gathered}
\big(\hat{\Op}^{(1)}_{per}-\l\big)v_\pm=0 \quad\text{in}\quad\Pi\setminus\om_0,
\qquad v_\pm=0\quad\text{on}\quad\p\Pi,
\\
[v_\pm]_0=\pm f_\pm\big|_{x_1=\pm 0}, \qquad \left[\frac{\p v_\pm}{\p\nu^{(1)}}\right]_0=\pm \frac{\p f_\pm}{\p\nu^{(1)}}\bigg|_{x_1=\pm 0}.
\end{gathered}
\end{equation}

Since $\l\notin\spec(\Op_{per}^{(k)})$, $k=1,2$, a statement similar to Lemma~\ref{lm5.4} holds for problem (\ref{4.24}). Namely, these problems are uniquely solvable. The operators $\cU^{(k)}_{per}(\l)$, $k=1,2$, mapping $f_\pm$ into the solutions are holomorphic in $\l\in\Xi$ as bounded operators from $\H^2(\Om_\pm)$ into $\H^2(\Pi_0^-)\oplus\H^2(\Pi_0^+)$. Moreover, for $X$ large enough the estimates hold:
\begin{equation}\label{4.25}
\|\cT_-\cS(-X)\cU^{(1)}_{per}(\l)\|_{\mW\to\H^2(\Om_-)}\leqslant Ce^{-c X},
\qquad \|\cT_+\cS(X)\cU^{(2)}_{per}(\l)\|_{\mW\to\H^2(\Om_+)}\leqslant Ce^{-c X},
\end{equation}
where $C$ and $c$ are some fixed positive constants independent of $\l\in\Xi$.

We first recover the function $u_1(x)$ as $x_1>a$ and  $u_2(x)$ as $x_1<-a$. In other words, we recover the functions $u_1(x_1-X_\ell^{(1)},x')$ and $u_2(x_1-X_\ell^{(2)},x')$ as $X_\ell^{(1)}+a<x_1<X_\ell^{(2)}-a$. In this zone, the coefficients of the operator $\Op_\ell$ are pure periodic and $\hat{\Op}_\ell$ coincides with $\hat{\Op}_{per}^{(1)}$, the needed functions $u_1$ and $u_2$ should solve the equation $(\hat{\Op}_{per}^{(1)}-\l)u_k=0$, $k=1,2$, should vanish on the $\p\Pi$ and they should also satisfy the boundary conditions
\begin{align*}
&(u_1(x_1-X_\ell^{(1)},x')+u_2(x_1-X_\ell^{(2)},x'))\big|_{x_1=X_\ell^{(1)}+a+0}= \psi(x)\big|_{x_1=X_\ell^{(1)}+a+0},
\\
&(u_1(x_1-X_\ell^{(1)},x')+u_2(x_1-X_\ell^{(2)},x'))\big|_{x_1=X_\ell^{(2)}-a-0}= \psi(x)\big|_{x_1=X_\ell^{(2)}-a-0},
\end{align*}
and the same condition for the first derivatives $\frac{\p\ }{\p\nu^{(1)}}$. The function $u_1$ should decay  as $x_1$ grows, while $u_2$ should decay as $x_1$ goes in the opposite direction.

The main idea how to find $u_1$ and $u_2$ is a follows. We first treat $u_2$ as an unknown function, and then $u_1=\psi-u_2$. We write the discussed equation and boundary conditions for $u_1$ at $x_1=X_\ell^{(1)}+a+0$ as   appropriate boundary value problem (\ref{4.24}) for $v_+=u_1$. The values of the function $\psi(x_1-X_\ell{(1)}-a)-v_+(x)$ at $x_1=X_\ell^{(2)}-X_\ell^{(1)}-a$ serve as the boundary conditions for the function $u_2=v_-$ in corresponding boundary value problem (\ref{4.24}). Once we know the function $u_2$, we
consider its values at $x_1=-(X_\ell^{(2)}-X_\ell^{(1)}-2a)$ and we recall that these values were initially considered as an unknown function. So, we arrive at certain functional equation for $u_2$, which turns out to be uniquely solvable. This determines uniquely both $u_1$ and $u_2$. We proceed to a rigorous realization of this idea.

We choose an arbitrary $f\in\H^2(\Om_+)$ and let
\begin{equation}\label{4.26}
v_+:=\cU^{(1)}_{per}(\l)\cS(-X_\ell^{(1)}-a)\psi-\cU^{(1)}_{per}(\l)f.
\end{equation}
In terms of this function, we define one more function $v_-$ as
\begin{equation}\label{4.27}
\begin{aligned}
v_-:=&\cU^{(2)}_{per}(\l)\cT_-\cS(-X_\ell^{(2)})\psi
-\cU^{(2)}_{per}(\l)\cT_- \cS(-X_\ell^{(2)}-X_\ell+2a) v_1
\\
=&\cU^{(2)}_{per}(\l)\cT_-\cS(-X_\ell^{(2)})\psi
-\cU^{(2)}_{per}(\l)\cT_- \cS(-X_\ell^{(2)}-X_\ell+2a) \cU^{(1)}_{per}(\l)\cS(-X_\ell^{(1)}-a)\psi
\\
&+\cU^{(2)}_{per}(\l)\cT_- \cS(-X_\ell^{(2)}-X_\ell+2a)\cU^{(1)}_{per}(\l)f.
\end{aligned}
\end{equation}
Then we postulate that
\begin{equation}\label{4.33}
f=\cT_+\cS(X_\ell^{(2)}-X_\ell^{(1)}-2a) v_-
\end{equation}
and in view of (\ref{4.27}) this gives rise to an equation:
\begin{equation}\label{4.28}
f-\cL(\l,X)f=h,\quad X:=X_{\ell}^{(2)}-X_\ell^{(1)}-2a,
\end{equation}
where
\begin{align*}
&
h:=\cT_+\cS(X)\cU^{(2)}_{per}(\l)\cT_-\cS(-X)\psi
-\cT_+\cS(X)\cU^{(2)}_{per}(\l)\cT_- \cS(-X) \cU^{(1)}_{per}(\l)\cS(-X_\ell^{(1)}-a)\psi,
\\
&\cL(\l,X):=\cT_+\cS(X)\cU^{(2)}_{per}(\l)\cT_- \cS(-X)\cU^{(1)}_{per}(\l).
\end{align*}
Thanks to estimates (\ref{4.25}), we see immediately that the operator $\cL(\l,X)$ is bounded in the space $\H^2(\Om_+)$ and its norm is exponentially small in $X$:
\begin{equation*}
\|\cL(\l,X)\|_{\H^2(\Om_+)\to\H^2(\Om_+)}\leqslant Ce^{-2cX}.
\end{equation*}
Hence, equation (\ref{4.28}) is uniquely solvable and this determines the function $f$ and the functions $v_\pm$ via formulae (\ref{4.26}), (\ref{4.27}). In view of the latter formulae, it is straightforward to check that the function
\begin{equation*}
v(x):=\left\{
\begin{aligned}
& v_+(x) &&\text{as}\quad x_1<0,
\\
& v_+(x)+v_-(x_1-X,x') &&\text{as}\quad 0<x_1<X,
\\
&v_-(x_1-X,x') &&\text{as}\quad x_1>X,
\end{aligned}
\right.
\end{equation*}
solves the boundary value problem
\begin{equation}\label{4.29}
\begin{aligned}
&\big(\hat{\Op}^{(1)}_{per}-\l\big)v=0 \quad\text{in}\quad\Pi\setminus(\om_0\cup\om_X),\quad
&& \hphantom{v}v=0\quad\text{on}\quad\p\Pi,
\\
&[v]_0=S(-X_\ell^{(1)}-a)\psi\big|_{x_1=+0}, && \left[\frac{\p v}{\p\nu^{(1)}}\right]_0=\frac{\p\ }{\p\nu^{(1)}}
S(-X_\ell^{(1)}-a)\psi\big|_{x_1=+0},
\\
&[v]_X=-S(-X_\ell^{(1)}-a)\psi\big|_{x_1=X-0},
&& \left[\frac{\p v}{\p\nu^{(1)}}\right]_X=-
\frac{\p\ }{\p\nu^{(1)}}S(-X_\ell^{(1)}-a)\psi\big|_{x_1=X-0},
\end{aligned}
\end{equation}
where
\begin{equation*}
\om_X:=\{X\}\times\om,\qquad [v]_X:=v\big|_{x_1=X+0}-v\big|_{x_1=X-0}.
\end{equation*}
This problem is uniquely solvable, which can be proved  as the same was done for problems (\ref{3.26}), (\ref{3.27}) in Lemma~\ref{lm5.4}. It is easy to see that the function
\begin{equation*}
x\mapsto \left\{
\begin{aligned}
&0 && \text{as}\quad x_1<0,
\\
\psi(x_1+&X_\ell^{(1)}+a,x')\quad   && \text{as}\quad 0<x_1<X,
\\
& 0 && \text{as}\quad x_1>X,
\end{aligned}
\right.
\end{equation*}
solves problem (\ref{4.29}). Hence, $\psi(x_1+X_\ell^{(1)}+a,x')=v(x)$ and therefore,
\begin{align}
&v_+(x)+v_-(x_1-X,x')=\psi(x_1+X_\ell^{(1)}+a,x') \quad \text{as}\quad x_1\in(0,X),\label{4.34}
\\
& v_+(x)=0 \quad \text{as}\quad x_1<0,\qquad v_-(x)=0 \quad \text{as}\quad x_1>0.\label{4.32}
\end{align}
Then we let
\begin{equation}\label{4.30}
u_1(x):=v_+(x_1-a,x')\quad\text{as}\quad x_1>a,\qquad u_2(x):=v_-(x_1+a,x')\quad\text{as}\quad x_1<-a.
\end{equation}
It remains to find the function  $u_1$ for $x_1\in(0,a)$ and the function $u_2$ for $x_1\in(-a,0)$. We define them as
\begin{equation}\label{4.31}
\begin{aligned}
&u_1(x):=\psi(x_1-X_\ell^{(1)},x')-u_2(x_1-X_\ell^{(2)}+X_\ell^{(1)},x')\quad\text{as}\quad x_1\in(0,a),
\\
&u_2(x):=\psi(x_1-X_\ell^{(2)},x')-u_1(x_1+X_\ell^{(2)}-X_\ell^{(1)},x')\quad\text{as}\quad x_1\in(-a,0),
\end{aligned}
\end{equation}
We observe that in the above definition of the function $u_1$, respectively, $u_2$, we employ the values of the function $u_2$, respectively, $u_1$, already defined in (\ref{4.30}).

It follows from (\ref{4.32}), (\ref{4.33}), (\ref{4.26}), (\ref{4.27}) that the function  $u_1$  defined in (\ref{4.30}), (\ref{4.31}) satisfies the identities:
\begin{equation*}
u_1(a+0,x')-u_1(a-0,x')=v_+(+0,x')-\psi\big(a-X_\ell^{(1)},x'\big)+ u_2\big(a-X_\ell^{(2)}+X_\ell^{(1)},x'\big)
=v_+(-0,x')=0.
\end{equation*}
In the same way we confirm that
\begin{gather*}
\frac{\p u_1}{\p\nu^{(1)}}(a+0,x')=\frac{\p u_1}{\p\nu^{(1)}}(a-0,x'),
\\
u_2(-a+0,x')=u_2(-a-0,x'),\qquad \frac{\p u_2}{\p\nu^{(1)}}(-a+0,x')=\frac{\p u_2}{\p\nu^{(1)}}(-a-0,x').
\end{gather*}
Hence, the found function  $u_1$ belong to $\H^2(\Pi_0^-)\oplus\H^2(\Pi_0^+)$, while the function $u_2$ belongs to $\H^2(\Pi_0^-)$. It also follows from (\ref{4.34}) and (\ref{4.23}) that
\begin{equation*}
u_1\big(x_1-X_\ell^{(1)},x'\big)+u_2\big(x_1-X_\ell^{(2)},x'\big)=\psi(x)\quad\text{as}\quad x_1<X_\ell^{(2)}.
\end{equation*}
The function $u_1$ also solves problem  (\ref{3.26}), (\ref{3.27}), (\ref{5.7d}), (\ref{5.7f}) with $g_1=(0\oplus\cT_+ \cS(X_\ell^{(2)}-X_\ell^{(1)})u_2)$. It also determines partially $g_2$, namely, $\cT_-g_2=\cT_-\cS(X_\ell^{(1)}-X_\ell^{(2)})u_1$. Other functions $u_j$ and $g_j$ can be recovered by repeating the above described procedure for $x_1\in(X_\ell^{(k)},X_\ell^{(k+1)})$, $k=2,\ldots,n-1$. This completes the proof of Lemma~\ref{lm6.1}.

\subsection{Reduction to   operator equation}\label{ss:reduc1}
Here we reduce problem  (\ref{2.14}) to some operator equation, which is more convenient for further purposes. We introduce a Hilbert space
\begin{align*}
\mG:=\Bigg\{\mathbf{g}=
\begin{pmatrix}
g_1
\\
\ldots
\\
g_n
\end{pmatrix},\  g_k\in\mW,\  k=1,\ldots,n,\ g_1\big|_{\Om_-}=0,\ g_n\big|_{\Om_+}=0\Bigg\}, \qquad
(\mathbf{g},\mathbf{h})_{\mG}=\sum\limits_{k=1}^{n}(g_k,h_k)_{\mW}.
\end{align*}
Then we choose an element $\mathbf{g}\in \mG$ and construct a function $\psi$ by formula (\ref{6.5}). We know by Lemma~\ref{lm6.1} that all nontrivial solutions of problem (\ref{2.14})
associated with resonances $\l\in\Xi$ are of form (\ref{6.5}), so, instead of finding $\psi$, we are going to find a corresponding $\mathbf{g}\in\mG$.

The function $\psi$ introduced by (\ref{6.5}) satisfies boundary conditions in   (\ref{2.14}) and possesses a needed behavior at infinity. We only need to confirm that it belongs to $\Hloc^2(\Pi)$ and solves the equation in  (\ref{2.14}). The former condition is ensured by the continuity in the trace sense of $\psi$ and $\frac{\p\psi\quad}{\p\nu^{(k)}}$ at $\{X_\ell^{(k)}\}\times\om$, $k=1,\ldots,n$. The equation is to be checked only as $|x_1-X_\ell^{(k)}|<a$ since outside these zones the equation is obviously satisfied. In view of the definition of the functions $u_k$ and formulae (\ref{4.22}) it is easy to see that both the belonging to $\Hloc^2(\Pi)$ and the validity of the equation hold once
\begin{align*}
&g_1=-\big(0\oplus \cT_+ u_2(\cdot+X_\ell^{(1)}-X_\ell^{(2)})\big),
\\
&g_k=-\big(\cT_- u_{k-1}(\cdot + X_\ell^{(k)}-X_\ell^{(k-1)})\oplus \cT_- u_{k+1}(\cdot + X_\ell^{(k)}-X_\ell^{(k+1)})\big),\quad k=2,\ldots,n-1,
\\
&g_n=-\big(\cT_- u_{n-1}(\cdot + X_\ell^{(n)}-X_\ell^{(n-1)})\oplus 0\big).
\end{align*}
 These identities can be rewritten as a system of operator equations
\begin{align}\label{6.2}
\begin{aligned}
&
g_1+\cT_{12}g_2=0,
\\
&
g_k+\cT_{kk-1}g_{k-1}+\cT_{kk+1}g_{k+1}=0,\quad k=2,\ldots,n-1,
\\
&
g_n+\cT_{nn-1}g_{n-1}=0,
\end{aligned}
\end{align}
where $\cT_{kj}$
are  operators in the space $\mW$  defined as
\begin{equation*}
\cT_{k\,k+1}(\l,\ell)g_{k+1}=\cT_+\mathcal{S}(X_{\ell}^{(k+1)}-X_{\ell}^{(k)})u_{k+1},\qquad
\cT_{k\,k-1}(\l,\ell)g_{k-1}=\cT_-\mathcal{S}(X_{\ell}^{(k)}-X_{\ell}^{(k-1)})u_{k-1},
\end{equation*}

In the space $\mG$ we introduce an  operator:
\begin{equation*}
\cT(\lambda, \ell):=
\begin{pmatrix}
0 & \cT_{12} & 0 & 0 & \ldots & 0 & 0 & 0 & 0
\\
\cT_{21} & 0 & \cT_{23} & 0 & \ldots & 0 & 0 & 0 & 0
\\
0 & \cT_{32} & 0 & \cT_{34} & \ldots & 0 & 0 & 0 & 0
\\
\vdots & \vdots & \vdots & \vdots& \ddots & \vdots & \vdots & \vdots & \vdots
\\
0 & 0 & 0 & 0 & \ldots & \cT_{n-2n-3} & 0  & \cT_{n-2n-1} & 0
\\
0 & 0 & 0 & 0 & \ldots & 0 & \cT_{n-1n-2} & 0 & \cT_{n-1n}
\\
0 & 0 & 0 & 0 & \ldots & 0 & 0 & \cT_{nn-1} & 0
\end{pmatrix}.
\end{equation*}
In terms of the above notation, system (\ref{6.2}) casts into the form:
\begin{equation}\label{6.3}
\mathbf{g}+\cT(\lambda,\ell)\mathbf{g}=0.
\end{equation}
The sought resonances of the operator $\Op_\ell$ are the value of   $\l\in\Xi$, for which the latter equation possesses non-trivial solutions.
The further study of this equation is based on the approach suggested  in \cite{AHP07}, \cite{MPAG07}, see also  \cite{DIB}. This approach will allow us to reduce the above problem on nontrivial solutions of   (\ref{6.3}) to searching zeroes of some holomorphic function.

\subsection{Reduction to  algebraic equations}\label{ss:reduc2}

We recall that $\l_0$ is an eigenvalue of some of the operators $\Op^{(k)}$, $k\in\{0,\ldots,n\}$, of multiplicities $N^{(k)}$ with associated orthonormalized in $L_2(\Pi)$ eigenfunctions $\phi_p^{(k)}$, $p=1,\ldots,N^{(k)}$; the identity $N^{(k)}=0$ corresponds to the case, when
 $\l_0$ is not an eigenvalue of the operator $\Op^{(k)}$. We denote
\begin{equation*}
\mathbf{\Phi}^{(k)}_{p,+}(\ell):=
\begin{pmatrix}
0
\\
\vdots
\\
0
\\
\cT_- \mathcal{S}\left(X_{\ell}^{(k)}-X_{\ell}^{(k+1)}\right)\phi^{(k)}_p
\\
0
\\
\vdots
\\
0
\end{pmatrix}\in \mG,\qquad p=1,\ldots,N^{(k)},
\end{equation*}
where the non-zero element stands as  $(k+1)$th position and
\begin{equation*}
\mathbf{\Phi}^{(k)}_{p,-}(\ell):=
\begin{pmatrix}
0
\\
\vdots
\\
0
\\
\cT_+ \mathcal{S}\left(X_{\ell}^{(k)}-X_{\ell}^{(k-1)}\right)\phi^{(k)}_p
\\
0
\\
\vdots
\\
0
\end{pmatrix}\in \mG,\qquad p=1,\ldots,N^{(k)},
\end{equation*}
where the non-zero element stands as  $(k-1)$th position. If $N^{(k)}=0$, the above vectors are supposed to be zero. We also denote
\begin{equation*}
\mathbf{\Phi}^{(1)}_p:=\mathbf{\Phi}_{p,+}^{(1)},\qquad \mathbf{\Phi}^{(n)}_p:=\mathbf{\Phi}_{p,-}^{(n)},\qquad \mathbf{\Phi}^{(k)}_p:=\mathbf{\Phi}_{p,-}^{(k)} + \mathbf{\Phi}_{p,+}^{(k)}.
\end{equation*}

By (\ref{2.25}), (\ref{2.22}), (\ref{2.26}), (\ref{2.27}), the estimate holds:
\begin{equation}\label{5.26}
\|\mathbf{\Phi}^{(k)}_p\|_{\mG} \leqslant C  \|\ell\|^\vk  e^{-\mr \la\ell\ra},
\end{equation}
where  $C>0$ are some constants independent of $\ell$.

In  view of Lemmata~\ref{lm5.4},~\ref{lm5.3b} and the definition of the functions  $u_k$, for $\l$ close to $\l_0$, the operator  $\cT(\l,\ell)$ admits the representation:
\begin{equation}\label{5.19}
\begin{aligned}
\cT(\l,l)\mathbf{g}=&\frac{1}{\l_0-\l}   \sum\limits_{k=1}^{n} \sum\limits_{p=1}^{N^{(k)}} \cC_p^{(k)}(\l)\mathbf{g}\, \mathbf{\Phi}_{p}^{(k)}(\ell)
+\cR(\l,\ell)\mathbf{g}.
\end{aligned}
\end{equation}
Here $\cC_p^{(k)}=\cC_p^{(k)}(\l)$ are the functionals on $\mG$ defined as
\begin{equation}\label{3.11}
\cC_p^{(k)}(\l)\mathbf{g}:=\cP_p^{(k)}(\l) g_k,
\end{equation}
and the writing  $\cC_{p}^{(k)}(\l)\mathbf{g}\,\mathbf{\Phi}_{p}^{(k)}(\l,\ell)$ denotes a usual scalar multiplication of the number $\cC_p^{(k)}(\l)\mathbf{g}$ and the vector $\mathbf{\Phi}_p^{(k)}(\l,\ell)$. The symbol $\cR=\cR(\l,\ell)$ stands for the following operator in $\mG$:
\begin{equation*}
\cR:=
\begin{pmatrix}
0 & \cR_{12} & 0 & 0 & \ldots & 0 & 0 & 0 & 0
\\
\cR_{21} & 0 & \cR_{23} & 0 & \ldots & 0 & 0 & 0 & 0
\\
0 & \cR_{32} & 0 & \cR_{34} & \ldots & 0 & 0 & 0 & 0
\\
\vdots & \vdots & \vdots & \vdots& \ddots & \vdots & \vdots & \vdots & \vdots
\\
0 & 0 & 0 & 0 & \ldots & \cR_{n-2n-3} & 0  & \cR_{n-2n-1} & 0
\\
0 & 0 & 0 & 0 & \ldots & 0 & \cR_{n-1n-2} & 0 & \cR_{n-1n}
\\
0 & 0 & 0 & 0 & \ldots & 0 & 0 & \cR_{nn-1} & 0
\end{pmatrix},%\l%abel{3.10}
\end{equation*}
where $\cR_{kj}=\cR_{kj}(\l,\ell)$ are the operators in  $\mW$ defined by the formulae
\begin{equation*}
\cR_{k\,k+1}:=\cT_+\mathcal{S}(X_{\ell}^{(k+1)}-X_{\ell}^{(k)})\cR^{(k+1)}(\l), \qquad
\cR_{k\,k-1}:=\cT_+\mathcal{S}(X_{\ell}^{(k-1)}-X_{\ell}^{(k)})\cR^{(k-1)}(\l),
\end{equation*}
with the operators $\cR^{(k\pm 1)}$  from representations (\ref{6.4a}).  
In view of estimates (\ref{3.3}), (\ref{3.29}),  the operators   $\cR_{kj}$ acting in  $L_2(\Pi_{a+1})$ possess an exponentially small norm and the same is true for the operator $\cR$ in the space $\mG$. Namely, there exist   $C>0$ and $c>0$ independent of $\l\in\Xi$  such that for all $\l\in\Xi$ the estimates hold
\begin{equation}\label{5.20}
\begin{aligned}
& \|\cR_{k-1\,k}(\l,\ell)\| \leqslant C |X_l^{(k)}-X_l^{(k-1)}|^\vk
e^{-(\mr-c|\IM\l|)\|\ell\|},
\\
&  \|\cR_{k+1\,k}(\l,\ell)\| \leqslant C |X_l^{(k+1)}-X_l^{(k)}|^\vk
e^{-(\mr-c|\IM\l|)\|\ell\|},
\\
&
\|\cR(\l,\ell)\|\leqslant C \|\ell\|^\vk e^{-(\mr-c|\IM\l|)\|\ell\|}
\end{aligned}
\end{equation}
for sufficiently large $\la\ell\ra$.
Moreover, the operator $\cR(\l,\ell)$ is holomorphic in  $\l\in\Xi$.
 In particular, this means that there exists an inverse operator
\begin{equation}\label{5.27}
\cQ=\cQ(\l,\ell):=\big(\cI+\cR(\l,\ell)\big)^{-1}
\end{equation}
holomorphic in $\l$, where $\cI$ is the identity mapping.

We substitute representation (\ref{5.19}) into equation (\ref{6.3}) and apply then the operator  $\cQ$:
\begin{equation}\label{5.21}
\mathbf{g} +
  \frac{1}{\l_0-\l}  \sum\limits_{k=1}^{n} \sum\limits_{p=1}^{N^{(k)}} \cC_p^{(k)}(\l)\mathbf{g}\, \cQ(\l,\ell)
\mathbf{\Phi}_p^{(k)}(\ell)=0.
\end{equation}
We apply the functionals  $\cC^{(r)}_j(\l)$ to the obtained equation:
\begin{equation}\label{5.22}
\cC^{(r)}_j(\l)\mathbf{g} +
  \frac{1}{\l_0-\l}   \sum\limits_{k=1}^{n} \sum\limits_{p=1}^{N^{(k)}} \cC_p^{(k)}(\l)\mathbf{g}\, \cC^{(r)}_j(\l)\cQ(\l,\ell)
\mathbf{\Phi}_{p}^{(k)}(\ell)=0.
\end{equation}
These equations are a system of linear algebraic equations with respect to unknown quantities
$\cC^{(r)}_j(\l)\mathbf{g}$. We observe first that only nontrivial solutions of system
(\ref{5.22}) can generate nontrivial solution of problem  (\ref{2.14}). Indeed, by equation
 (\ref{5.21}), the trivial solution $\cC_p^{(k)}(\l)\mathbf{g}=0$ gives $\mathbf{g}=0$. In its turn, this means that the corresponding functions $u_k$ are also trivial and the same is true for the function  $\psi$ defined by formula (\ref{6.5}).

To study the existence of nontrivial solutions to system (\ref{5.22}), we rewrite it first to a matrix form. We introduce the vector of unknowns in a block form:
\begin{align*}
&\sC:=
\begin{pmatrix}
\sC_1
\\
\vdots
\\
\sC_n
\end{pmatrix},
\qquad
\sC_k=\sC_k(\l):=
\begin{pmatrix}
\cC_p^{(k)}(\l)\mathbf{g}
\end{pmatrix}_{p=1,\ldots,N^{(k)}}
\end{align*}
Hereinafter the vectors of the form $(\cdot)_{p=1,\ldots,N^{(k)}}$ are treated as the vector  columns. If   $N^{(k)}=0$, the corresponding column in the above formulae is absent. The total size of the introduced column is equal to $N$ defined by (\ref{2.30}).

We introduce a matrix $\sA=\sA(\l,\ell)$ of size $N\times N$. This matrix has a block form:
\begin{equation}\label{5.30}
\sA:=
\begin{pmatrix}
\sA_{11} & \ldots & \sA_{1n}
\\
\vdots && \vdots
\\
\sA_{n1} & \ldots & \sA_{nn}
\end{pmatrix}.
\end{equation}
Each  block $\sA_{rk}=\sA_{rk}(\l,\ell)$ is of the size $N^{(r)}\times N^{(k)}$. The blocks are defined as
\begin{equation}
\label{5.31}
\sA_{rk}(\l,\ell):=
\begin{pmatrix}
\cC_j^{(r)}(\l)\cQ(\l,\ell)  \mathbf{\Phi}_{p}^{(k)}
\end{pmatrix}^{j=1,\ldots,N^{(r)}}_{p=1,\ldots,N^{(k)}}.
\end{equation}

In view of the introduced notations, equations (\ref{5.22}) are rewritten to the matrix one:
\begin{equation}\label{5.24}
\bigg(\sE+\frac{1}{\l_0-\l}\sA(\l,\ell)\bigg)\sC=0.
\end{equation}
By the Cramer's rule, the existence of nontrivial solution  is equivalent to the equation
\begin{equation}\label{5.25}
\det\bigg(\big(\l-\l_0\big) \sE -\sA(\l,\ell)\bigg)=0,
\end{equation}
where $\sE$ is the unit matrix. This identity is the equation for the sought resonances of the operator  $\Op_\ell$. For each root $\l=\l(\ell)$  of this equation, the corresponding nontrivial solution of system  (\ref{5.24}) generates the solution  $\mathbf{g}$ of equation (\ref{5.21}). In its turn, by formula  (\ref{6.5}), this solution generates a solution to problem (\ref{2.14}). Thus, we need to study the existence and behaviour of roots of equation  (\ref{5.25}). This will be done in the next section.

\section{Solvability of equation (\ref{5.25}) and behavior of its roots}\label{s:proof}

In this section we complete the proof of Theorem~\ref{th2.2}.
Our strategy is as follows. First we establish a preliminary rough estimate  for the matrix $\sA$, which allows us to prove the solvability of equation (\ref{5.25}) and to localize the roots, that is, to prove that all roots are contained in a circle of an exponentially small radius centered at $\l_0$.  The next step is devoted to describing the asymptotics of the matrix $\sA$ as $\ell\to\infty$. Employing this asymptotics, in  a final step we find leading terms in the asymptotics of the roots.

Throughout this section, by $C$ we denote various inessential constants independent of sufficiently large $\ell$ and $\l\in\Xi$. By  $B_r(z)$ we denote the ball in the complex plane of a radius $r$ centered at a point $z$.

\subsection{Solvability}

We begin with an  obvious property implied by Lemmata~\ref{lm5.4},~\ref{lm5.3b}, namely, a holomorphic dependence of the matrix $\sA$ in $\l\in\Xi$. This implies that the function in the left hand side in (\ref{5.25}) is holomorphic in $\l$.

The third estimate in (\ref{5.20}) and definition (\ref{5.27}) of the operator $\cQ$ yield immediately that this operator is bounded uniformly in $\l\in\Xi$ and sufficiently large $\la\ell\ra$. Then inequality (\ref{5.26}) and  definition (\ref{5.30}), (\ref{5.31}) of the matrix $\sA$ yield the following bound for the matrix $\sA$:
\begin{equation}\label{5.32}
\|\sA(\l,\ell)\|\leqslant C \eta(\ell).
\end{equation}
Here as a norm for the matrix $\sA$, we choose the maximal among the absolute values of its entries.

We calculate the determinant in the left hand side in equation (\ref{5.25})   rewriting this equation as
\begin{equation}\label{5.33}
(\l-\l_0)^N+\mathrm{F}(\l,\ell)=0,\qquad
\mathrm{F}(\l,\ell):=\sum\limits_{i=0}^{N-1}\mathrm{F}_i(\l,\ell)(\l-\l_0)^i,
\end{equation}
where $\mathrm{F}_i(\l,\ell)$ are some functions holomorphic in $\l\in\Xi$; thanks to (\ref{5.32}), these functions obey the estimates
\begin{equation}\label{5.34}
|\mathrm{F}_i(\l,\ell)|\leqslant C \eta^{(N-i)}(\ell).
\end{equation}

It follows from (\ref{5.33}), (\ref{5.34}) that each root $\l\in\Xi$ of equation (\ref{5.25}) satisfies the estimate
\begin{equation}\label{5.36}
|\l-\l_0|\leqslant C\eta^\frac{1}{N},
\end{equation}
and hence, it converges to $\l_0$ as $\ell\to\infty$. Then we consider the circle $B_{\d}(\l_0)\subset\Xi$ of a fixed radius $\d$ and by (\ref{5.34}) we see that
\begin{equation*}%\l%abel{5.35}
|\mathrm{F}_i(\l,\ell)|<|\l-\l_0|^N=\d^n\quad\text{as}\quad \l\in\p B_\d(\l_0).
\end{equation*}
This estimate and the aforementioned holomorphy of $\mathrm{F}$ in $\l\in\Xi$ allow us to apply the Rouch\'e theorem and to conclude that the function in the left hand side in equation (\ref{5.25}) possesses exactly the same amount of the zeroes in $B_\d(\l_0)$, counting their orders,  as the function $\l\mapsto (\l-\l_0)^N$ does. Hence, equation (\ref{5.25}) has exactly $N$ roots in $\Xi$ counting their orders; by (\ref{5.36}), all these roots tend to $\l_0$ as $\ell\to\infty$.

Finally, we are going to improve estimate (\ref{5.36}). We consider the circle $B_{c\eta(\ell)}(\l_0)$, where $c:=2+2C$, where $C$ is from   (\ref{5.34}). By (\ref{5.34}), on the boundary of this circle we have the estimate:
%\begin{equation}\l%abel{5.37}
\begin{equation*}
|\mathrm{F}(\l,\ell)|\leqslant \sum\limits_{i=0}^{N-1} |\mathrm{F}_i(\l,\ell)| c^{N-i}\eta^{N-i}(\ell)\leqslant \frac{c^{N-1}C}{1-c^{-1}}
 \eta^N(\ell)
 \leqslant  
 2C c^{N-1}\eta^N(\ell)<c^N\eta^N(\ell)=|\l-\l_0|^N.
\end{equation*}
%\end{equation}
Hence, we can apply the Rouch\'e theorem once again and we see that equation (\ref{5.25}) has exactly $N$ roots, counting their orders, in the circle $B_{c\eta(\ell)}(\l_0)$. This means that all these roots satisfy the estimate:
\begin{equation*}%\l%abel{5.38}
|\l-\l_0|\leqslant c\eta(\ell).
\end{equation*}
This is the desired estimate for the roots of equation (\ref{5.25}).

\subsection{Asymptotics for  matrix $\sA$}

In the present subsection we find an asymptotics for the matrix   $\sA$ as $\ell\to\infty$. Since all roots of equation (\ref{5.25}) are located in the circle $B_{c\eta(\ell)}(\l_0)$, in what follows we consider only $\l\in B_{c\eta(\ell)}(\l_0)$.

Estimate  (\ref{5.20}) and definition (\ref{5.27}) of the operator $\cQ$ allow us to expand the latter operator into the standard Neumann series, which implies, in particular, the representations:
\begin{equation*}%\l%abel{5.28a}
\cQ(\l,\ell)=\cI -\cR(\l,\ell)\cQ(\l,\ell),
\qquad
\|\cR(\l,\ell)\cQ(\l,\ell)\|\leqslant C \eta(\ell),
\end{equation*}
where
the operator   $\cR(\l,\ell)\cQ(\l,\ell)$  is holomorphic in $\l\in B_{c\eta(\ell)}(\l_0)$.
We also observe that since    $\l\in B_{c\eta(\ell)}(\l_0)$, it follows from the definition of the functionals $\cC^{(k)}_j$ that
\begin{equation*}
\big\|\cC^{(k)}_j(\l)-\cC^{(k)}_j(\l_0)\big\|\leqslant C\eta(\ell).
\end{equation*}
Hence, in view of (\ref{5.26}) and by the definition of the matrix $\sA$, we infer that it satisfies the representation
\begin{equation}\label{3.6}
\sA(\l,\ell)=\sB(\ell)+\sA_1(\l,\ell),
\end{equation}
where
\begin{equation*}
\sB:=
\begin{pmatrix}
\sB_{11} & \ldots & \sB_{1n}
\\
\vdots && \vdots
\\
\sB_{n1} & \ldots & \sB_{nn}
\end{pmatrix}, \qquad
\sB_{rk}(\ell):=
\begin{pmatrix}
\cC_j^{(r)}(\l_0)\mathbf{\Phi}_{p}^{(k)}
\end{pmatrix}^{j=1,\ldots,N^{(r)}}_{p=1,\ldots,N^{(k)}},
\end{equation*}
and $\sA_1$ is some matrix holomorphic in $\l\in\overline{B_{c\eta(\ell)}(\l_0)}$ obeying the estimate:
\begin{equation}\label{3.7}
\|\sA_1(\l,\ell)\|\leqslant C\eta^2(\ell).
\end{equation}

According \cite[Ch. 5, Sect. 1.3]{NP}, the functions $\vp^{(k,\pm)}_{is}$ introduced in (\ref{2.25}) solve boundary value problems
\begin{equation}\label{5.44}
(\hat{\Op}_{per}^{(k)}-\l_0)\vp^{(k,\pm)}_{is}=0\quad\text{in}\quad \Pi,\qquad \vp^{(k,\pm)}_{is}=0\quad\text{on}\quad \p\Pi.
\end{equation}

We first prove an auxiliary lemma.

\begin{lemma}\label{lm5.8}
The identities
\begin{equation}\label{5.48}
K_{isqt}^{(k)}=\pm\lim\limits_{N\to\pm\infty} \int\limits_{\om} \left(\overline{\vp_{qt}^{(k,-)}}\frac{\p\vp_{is}^{(k,+)}}{\p\nu^{(k)}}
-\vp_{is}^{(k,+)} \overline{\frac{\p\vp_{qt}^{(k,-)}
}{\p\nu^{(k)}}}
\right)\Bigg|_{x_1=N T^{(k)}}\di x',
\end{equation}
hold, where the limits in the right hand sides
are finite and are independent of the choice of the sign in their definition.
\end{lemma}

\begin{proof}
We choose an  arbitrary natural  $m$ large enough and integrate twice by parts in the following integral:
\begin{align*}
0=\int\limits_{\Pi_{m T^{(k)}}} \overline{\vp_{qt}^{(k,-)}} \big(\hat{\Op}_{per}^{(k)}-\l_0\big) \vp_{is}^{(k,+)}\di x
=&
\int\limits_{\om} \left(\overline{\vp_{qt}^{(k,-)}}\frac{\p\vp_{is}^{(k,+)}}{\p\nu^{(k)}}
-\vp_{is}^{(k,+)} \overline{\frac{\p\vp_{qt}^{(k,-)}
}{\p\nu^{(k)}}}
\right)\Bigg|_{x_1=m T^{(k)}}\di x'
\\
&+ \int\limits_{\om} \left(\overline{\vp_{qt}^{(k,-)}}\frac{\p\vp_{is}^{(k,+)}}{\p\nu^{(k)}}
-\vp_{is}^{(k,+)} \overline{\frac{\p\vp_{qt}^{(k,-)}
}{\p\nu^{(k)}}}
\right)\Bigg|_{x_1=-m T^{(k)}}\di x'.
\end{align*}
This implies that if the limits in the right hand side in  (\ref{5.48}) exist, then they are independent of the choice of the sign.

Let  $\xi_3=\xi_3(x_1)$ be an infinitely differentiable cut-off function  equalling to one as $x_1>2$ and vanishing as $x_1\leqslant 1$.
By problems (\ref{5.44}), the integrand in the integral
\begin{equation*}
\int\limits_{\Pi_{m T^{(k)}}} \overline{\vp_{qt}^{(k,-)}} \big(\hat{\Op}_{per}^{(k)}-\l_0\big) \chi_3\vp_{is}^{(k,+)}\di x
\end{equation*}
is compactly supported and this is why the integral is independent on $m$ large enough. Integrating by parts and bearing in mind problems   (\ref{5.44}) and definition (\ref{2.22}) of the functions $\vp_{is}^{(k,\pm)}$, for sufficiently large $m$ we obtain:
\begin{equation}\label{5.47}
\begin{aligned}
\int\limits_{\Pi_{m T^{(k)}}} \overline{\vp_{qt}^{(k,-)}} \big(\hat{\Op}_{per}^{(k)}-\l_0\big) \xi_3\vp_{is}^{(k,+)}\di x
&=  \int\limits_{\om}  \left(\overline{\vp_{qt}^{(k,-)}}\frac{\p\vp_{is}^{(k,+)}}{\p\nu^{(k)}}
- \vp_{is}^{(k,+)}\overline{\frac{\p \vp_{qt}^{(k,-)}
}{\p\nu^{(k)}}}
\right)\Bigg|_{x_1=m T^{(k)}}\di x'
\\
&= e^{\iu m\big( \mr_{i}^{(k,+)}- \overline{\mr_{q}^{(k,-)}}\big)T^{(k)}
}
\int\limits_{\om} \Bigg(\overline{\tilde{\vp}_{qt}^{(k,-)}}\frac{\p\tilde{\vp}_{is}^{(k,+)}}{\p\nu^{(k)}}
- \tilde{\vp}_{is}^{(k,+)}\frac{\p\overline{\tilde{\vp}_{qt}^{(k,-)}
}}{\p\nu^{(k)}}
\\
&\hphantom{=,} + \iu\big(\mr_{i}^{(k,+)} -\overline{\mr_{q}^{(k,-)}}\big)
\tilde{\vp}_{is}^{(k,+)}\overline{\tilde{\vp}_{qt}^{(k,-)}}
\frac{\p x_1}{\p\nu^{(k)}}
\Bigg)\Bigg|_{x_1= m T^{(k)}}\di x'.
\end{aligned}
\end{equation}
By definition   (\ref{2.22}) of the functions $\tilde{\vp}_{is}^{(k,\pm)}$ and the periodicity of the functions $\Phi_{is}^{(k,\pm)}$, the latter integral in  (\ref{5.47}) depends polynomially on  $m$. We also recall that $\overline{\mr_{q}^{(k,-)}}=\mr_{q}^{(k,+)}$.

Let $\mr_{i}^{(k,+)}\ne\mr_{q}^{(k,+)}$. Then the exponent $e^{ m\iu( \mr_{i}^{(k,+)}-\overline{\mr_{q}^{(k,-)}})T^{(k)}}=e^{\iu m(\mr_{i}^{(k,+)}-\mr_{q}^{(k,+)}) T^{(k)}}$ oscillates in  $m$; at that, it can increases or decreases in  $m$. In this case, identities (\ref{5.47}) are possible for all sufficiently large $m$ only as all integrals in these identities vanish. This proves formulae (\ref{5.45}).

Let $\mr_{i}^{(k,+)}=\mr_{q}^{(k,+)}$. In this case the exponent in the latter integral in  (\ref{5.47}) disappears. The remaining integral is polynomial in  $m$ and identities (\ref{5.47}) are possible only as this polynomial degenerates into its free coefficient. It is easy confirm that this coefficient is exactly the right hand side in formula  (\ref{5.46}). The proof is complete.
\end{proof}

All components of the vector $\mathbf{\Phi}_p^{(k)}(\ell)$ are zero except those at  $(k-1)$th and $(k+1)$th positions.  Hence, by  definition (\ref{3.11}) of the functionals   $\cC_j^{(r)}$
we see immediately that
\begin{equation}\label{5.61}
\sB_{rk}(\ell)\equiv0\quad \text{as}\quad |r-k|\ne 1.
\end{equation}

\begin{lemma}\label{lm5.6b}
The matrices  $\sA_{k\pm1\,k}$ satisfy the representations
\begin{align}\label{5.41}
\sA_{k\pm1\,k}(\l_0,\ell)=\sAo_{k\pm1\,k}(\ell) + O\big(e^{-\g\la\ell\ra}\big),\quad \ell\to\infty,
\end{align}
where, we recall, the matrices $\sAo_{k\pm1\,k}(\ell)$ were defined in (\ref{5.55}).
\end{lemma}

\begin{proof}
The entries of the matrices $\sAo_{k\pm1\,k}$ are determined  by the  quantities  $\cC_j^{(k\pm1)}(\l_0) \mathbf{\Phi}_p^{(k)}(\ell)$. Let us find out their asymptotic behavior. By formulae  (\ref{2.22}), (\ref{2.26}) we obtain:
%\begin{equation}\l%abel{5.50}
\begin{align*}
\cT_-\cS(X_\ell^{(k)}-X_\ell^{(k+1)})\phi_p^{(k)}
=&\sum\limits_{i=1}^{J^{(k)}}
e^{\iu\mathfrak{r}^{(k,+)}_{i} (X_\ell^{(k+1)}-X_\ell^{(k)})}
\sum\limits_{s=0}^{\vk^{(k)}_{i}-1} \b_{pis}^{(k,+)}(X_\ell^{(k+1)}-X_\ell^{(k)})
\cT_- \vp^{(k,+)}_{is}
\\
&+O\big(e^{-\g |X_\ell^{(k+1)}-X_\ell^{(k)}}|\big).
\end{align*}
%\end{equation}
In the same way we find
%\begin{equation}\l%abel{5.51}
\begin{align*}
\cT_+\cS(X_\ell^{(k+1)}-X_\ell^{(k)})\phi_p^{(k+1)}
=&\sum\limits_{i=1}^{J^{(k)}}
e^{\iu\mathfrak{r}^{(k,-)}_{i} (X_\ell^{(k+1)}-X_\ell^{(k)})}
\sum\limits_{s=0}^{\vk^{(k)}_{i}-1} \b_{pis}^{(k,-)}(X_\ell^{(k+1)}-X_\ell^{(k)})
\cT_+\vp^{(k,-)}_{is}
\\
&
+O\big(e^{- \g |X_\ell^{(k+1)}-X_\ell^{(k)}|}\big).
\end{align*}
%\end{equation}
It follows from the obtained formulae, identities (\ref{2.25}), definition (\ref{3.11}) of the functional $\cC^{(r)}_j$, and Lemma~\ref{lm5.8} that
\begin{equation}\label{5.52}
\begin{aligned}
\cC_j^{(k+1)}(\l_0)\mathbf{\Phi}^{(k)}_{p,+}(\ell)=
\sum\limits_{i=1}^{J^{(k)}}
&e^{\iu\mathfrak{r}^{(k,+)}_{i} (X_\ell^{(k+1)}-X_\ell^{(k)})}
\sum\limits_{s=0}^{\vk^{(k)}_{i}-1} \b_{pis}^{(k,+)}(X_\ell^{(k+1)}-X_\ell^{(k)})
\\
&\cdot\cP_j^{(k+1)}\cF^{(k+1)}(\l_0)(\cT_+ \vp^{(k,+)}_{is}\oplus 0)
+O\big(e^{-\g |X_\ell^{(k+1)}-X_\ell^{(k)}|}\big),
\end{aligned}
\end{equation}
According (\ref{3.28}), the quantity $\cP_p^{(k+1)}\cF^{(k+1)}(\l_0)(\cT_+ \vp^{(k,+)}_{is}\oplus 0)$ is given by formula
\begin{align*}
\cP_j^{(k+1)}
(\l_0)(\cT_+ \vp^{(k,+)}_{is}\oplus 0)=&\int\limits_{\Om_-}  \overline{\phi_j^{(k+1)}}(\Op^{(k+1)}-\l_0) \vp^{(k,+)}_{is}\di x
\\
&+\int\limits_{\om_0} \left(
\overline{\phi_j^{(k+1)}}\frac{\p \vp^{(k,+)}_{is}}{\p\nu^{(k+1)}}- \vp^{(k,+)}_{is} \overline{\frac{\p \phi_p^{(k+1)} }{\p\nu^{(k+1)}}}\right)\di s.
\end{align*}
As $x_1<-a$, thanks to the definition of the function $\vp^{(k,+)}_{is}$, the equation holds
\begin{equation*}
(\Op^{(k+1)}-\l_0) \vp^{(k,+)}_{is}=(\Op^{(k+1)}_{per}-\l_0) \vp^{(k,+)}_{is}=0
\end{equation*}
and this is why we can integrate parts as in (\ref{5.47}):
%\begin{equation}\l%abel{5.40a}
\begin{align*}
\cP_j^{(k+1)}(\l_0)(\cT_+ \vp^{(k,+)}_{is}\oplus 0)=&\lim\limits_{m\to-\infty}\int\limits_{(-m T^{(k)}, 0)\times\om}  \overline{\phi_j^{(k+1)}}(\Op^{(k+1)}-\l_0) \vp^{(k,+)}_{is}\di x
\\
&+\int\limits_{\om_0} \left(
\overline{\phi_j^{(k+1)}}\frac{\p \vp^{(k,+)}_{is}}{\p\nu^{(k+1)}}- \vp^{(k,+)}_{is} \overline{\frac{\p \phi_j^{(k+1)} }{\p\nu^{(k+1)}}}\right)\di s
\\
=&\lim\limits_{m\to-\infty} \int\limits_{\om}  \left(\vp^{(k,+)}_{is} \overline{\frac{\p\phi_j^{(k+1)}}{\p\nu} }- \overline{\phi_j^{(k+1)}}\frac{\p\vp^{(k,+)}_{is}}{\p\nu}
\right)\Bigg|_{x_1=m T^{(k)}}\di x'
\\
= &\sum\limits_{q=1}^{J^{(k)}}\sum\limits_{t=0}^{\vk_{q}^{(k)}-1} \overline{\a_{jqt}^{(k,-)}}
\lim\limits_{m\to-\infty} \int\limits_{\om}  \left(\vp^{(k,+)}_{is} \overline{\frac{\p \vp_{qt}^{(k,-)}}{\p\nu} }- \overline{\vp_{qt}^{(k,-)}}\frac{\p\vp^{(k,+)}_{is}}{\p\nu}
\right)\Bigg|_{x_1=m T^{(k)}}\di x'
\\
= &\sum\limits_{q=1}^{J^{(k)}}\sum\limits_{t=0}^{\vk_{q}^{(k)}-1} \overline{\a_{jqt}^{(k,-)}}
K_{isqt}^{(k)}.
\end{align*}
%\end{equation}
Substituting these identities into relations (\ref{5.52}), we arrive at  a final formula:
\begin{equation*}%\l%abel{5.53}
\cC_j^{(k+1)}(\l_0)\mathbf{\Phi}^{(k)}_{p,+}(\ell)=
\cAo_{jp}^{(k,+)}(\ell)
+
O\big(e^{-\g|X_\ell^{(k+1)}-X_\ell^{(k)}|}\big).
\end{equation*}
In the same way we confirm that
\begin{equation*}
\cP_j^{(k)}
(\l_0)
(0\oplus\cT_- \vp^{(k,-)}_{is})
= \sum\limits_{q=1}^{J^{(k)}}\sum\limits_{t=0}^{\vk_{q}^{(k)}-1} \overline{\a_{jqt}^{(k,+)}}
K_{qtis}^{(k)},
\end{equation*}
and
\begin{equation*}
\cC_j^{(k)}(\l_0)
\mathbf{\Phi}^{(k+1)}_{p,-}(\ell)=
\cAo_{jp}^{(k,-)}(\ell)
+O\big(e^{- \g  |X_\ell^{(k+1)}-X_\ell^{(k)}|}\big).
\end{equation*}
The obtained formulae yield  (\ref{5.41}), (\ref{5.55}). The proof is complete.
\end{proof}

\subsection{Asymptotics for the roots}\label{ss:asrt}

In this subsection we find the asymptotics for the roots of equation (\ref{5.25}). We begin with observing that identity (\ref{5.61}) and Lemma~\ref{lm5.6b} imply
\begin{equation*}
\sB(\ell)=\sAo(\ell)+O\big(e^{-\g\la\ell\ra}\big).
\end{equation*}
Then by (\ref{3.6}), (\ref{3.7}) we infer that
\begin{equation}\label{5.62}
\sA(\l,\ell)=\sAo(\ell)+\sA_2(\l,\ell),
\end{equation}
where $\sA_2(\l,\ell)$ is a holomorphic in $\l\in B_{c\eta(\ell)}(\l_0)$ matrix obeying the estimate
\begin{equation*}
\|\sA_2(\l,\ell)\|\leqslant C e^{-\g\la\ell\ra}.
\end{equation*}
Having this estimate and identity (\ref{5.62}) in mind as well as the fact that $\L_j$ are the eigenvalues of the matrix $\sAo$, we rewrite equation (\ref{5.25}) as
\begin{equation*}%\l%abel{5.63}
\prod\limits_{j=1}^{N}\big(\l-\L_j(\ell)\big) + G(\l,\ell)=0,
\end{equation*}
where $G$ is a holomorphic in $\l\in B_{c\eta(\ell)}(\l_0)$ function obeying the estimate
\begin{equation}\label{5.64}
|G(\l,\ell)|\leqslant  C e^{-\g\la\ell\ra}\eta^{N-1}(\ell);
\end{equation}
the factor $\eta^{N-1}$ appears in the latter estimate since $|\l-\l_0|<c\eta(\ell)$.

It follows from Lemma~\ref{lm5.6b} that the matrix $\sAo$ is of order $O(\eta)$ and this estimate is order sharp. Hence, the same is true for its eigenvalues.

We choose a group $L_p$ and take one of the eigenvalues $\L_j$ in this group. We consider the ball $B_{2\tilde{c}\vt}(\L_j)$, where $\tilde{c}>1$ is some fixed constant to be chosen later, and
$\vt(\ell):=\eta^{1-\frac{1}{N}}(\ell)e^{-\frac{\g}{N}\la \ell\ra}$. We assume that the constant
$\tilde{c}$ is such that $\{\L_i\}_{i\in L_p}\subset B_{2\tilde{c}\vt}(\L_j)$ and the distances from $\L_j$ to the boundary $\p B_{2\tilde{c}\vt}(\L_j)$ is at least $\tilde{c}\vt$. Hence,
\begin{equation}\label{5.65}
\prod\limits_{i\in L_p} |\l-\L_i|\geqslant (\tilde{c}\vt)^{|L_p|}\quad\text{as}\quad \l\in\p B_{2\tilde{c}\vt}(\L_j).
\end{equation}
It follows from (\ref{2.32}) that for $\L_i\notin L_p$ a similar estimate holds:
\begin{equation}
\prod\limits_{i\notin L_p} |\l-\L_i|\geqslant \left(\frac{\mu\vt}{2}\right)^{N-|L_p|}\quad\text{as}\quad \l\in\p B_{2\tilde{c}\vt}(\L_j).
\end{equation}
This inequality and (\ref{5.65}) imply:
\begin{equation*}%\l%abel{5.66}
\prod\limits_{i\in\{1,\ldots,N\}} |\l-\L_i|\geqslant \tilde{c}\vt^{N}=\tilde{c}
e^{-\g\la\ell\ra}\eta^{N-1}(\ell) \quad\text{as}\quad \l\in\p B_{2\tilde{c}\vt}(\L_j).
\end{equation*}
In view of this estimate and (\ref{5.64}) we see that choosing $\tilde{c}=2C$, we get
\begin{equation*}
\bigg|\prod\limits_{i\in\{1,\ldots,N\}} (\l-\L_i)\bigg|>|G(\l,\ell)|\quad\text{as}\quad \l\in\p B_{2\tilde{c}\vt}(\L_j)
\end{equation*}
and by Rouch\'e theorem, equation (\ref{5.25}) has exactly the same number of zeroes in $B_{2\tilde{c}\vt}(\L_j)$ as the function $\l\mapsto \prod\limits_{i\in\{1,\ldots,N\}} (\l-\L_i)$. The zeroes of the latter function in $B_{2\tilde{c}\vt}(\L_j)$ are exactly $\L_i$, $i\in L_p$. Hence, equation (\ref{5.25}) has the same number of roots counting their orders  in
$B_{2\tilde{c}\vt}(\L_j)$. All these roots satisfy   $|\l(\ell)-\L_j(\ell)|<2\tilde{c}\vt$ and this proves asymptotics (\ref{2.33}). In a general situation, $\l(\ell)$ is complex-valued  and in this case its imaginary part is negative, since otherwise the corresponding non-trivial solution to problem (\ref{2.14}) would be an eigenfunction associated with a complex-valued eigenvalue, what is impossible. If the root $\l(\ell)$ is real, then the associated non-trivial solution to problem (\ref{2.14}) can be an eigenfunction of the operator $\Op_\ell$ but this situation still fits our definition of the resonance.
The proof of Theorem~\ref{th2.2} is complete.

\section*{Acknowledgements}

The authors thank A.A. Fedotov and S.A. Nazarov for useful comments and discussion some aspects of the work.

The results presented in Sections~\ref{ss:reduc1},~\ref{ss:reduc2},~\ref{s:proof} were financially supported by Russian Science Foundation (grant no. 17-11-01004).

\end{document}